\documentclass[12pt]{amsart}
\usepackage{palatino}
\usepackage{amscd,amssymb}
\usepackage{float}
\usepackage{graphicx}

\usepackage{amsmath,amssymb,amsthm,amsfonts}
\usepackage{mathrsfs}
\usepackage[margin=1in]{geometry}
\usepackage{xcolor}
\usepackage{hyperref}

\newcommand{\query}[1]{%
  \ifmmode
    \text{\footnotesize\textcolor{red}{[?: #1]}}%
  \else
    {\footnotesize\textcolor{red}{[?: #1]}}%
  \fi
}

\newtheorem{theorem}{Theorem} 
\newtheorem{proposition}[equation]  {Proposition}
\newtheorem{lemma}[equation]  {Lemma}
\newtheorem{corollary}[equation]  {Corollary}
\theoremstyle{definition}
\newtheorem{definition}[equation]  {Definition}

\newtheorem{question}[equation]  {Question}

\newcommand{\Mbar}{\overline{M}}
\newcommand{\Nbar}{\overline{N}}
\newcommand{\Pbar}{\overline{P}}
\newcommand{\Qbar}{\overline{Q}}

\newcommand{\C}{\mathbb{C}}
\newcommand{\R}{\mathbb{R}}
\newcommand{\Z}{\mathbb{Z}}
\newcommand{\sla}{\mathfrak{sl}}
\newcommand{\cL}{\mathcal{L}}
\usepackage{float}
\usepackage{graphicx}

\usepackage[cmtip,matrix,arrow]{xy}

\usepackage{graphicx}
\numberwithin{equation}{section}

\newtheorem {Question}[equation]      {Question}

 \newtheorem{zremark}[equation] {Remark}
\newenvironment{Remark} {\par\footnotesize\zremark}{~\\}{\endzremark}
\newcommand{\pr} {\smallskip\noindent{\bf Proof\,\,}}

\makeatletter
\newenvironment{reptheorem}[1]
{\def\@thmtitle{Theorem~\ref{#1}}%
\trivlist\item[\hskip\labelsep\bfseries\@thmtitle.]\itshape}
{\endtrivlist}
\makeatother

\makeatletter
\newenvironment{repprop}[1]
{\def\@thmtitle{Proposition~\ref{#1}}%
\trivlist\item[\hskip\labelsep\bfseries\@thmtitle.]\itshape}
{\endtrivlist}
\makeatother

\makeatletter
\newenvironment{repcorollary}[1]
{\def\@thmtitle{Corollary~\ref{#1}}%
\trivlist\item[\hskip\labelsep\bfseries\@thmtitle.]\itshape}
{\endtrivlist}
\makeatother
 
\newcommand     {\mute}[2] {}
\newcommand     {\printname}[1] {}

\newcommand{\labell}[1] {\label{#1}\printname{#1}}

\def \hc {\hat{c}}

\def    \ker    {\operatorname{ker}}

\def    \to     {\longrightarrow}

\def    \C      {{\mathbb C}}
\def    \R      {{\mathbb R}}

\def    \Z      {{\mathbb Z}}

\def    \cP    {{\mathcal P}}

\def    \pr     {\operatorname{pr}}

\title[The Dual Category]{The Dual DG Category to Weinstein's symplectic "category" and applications to Geometric Quantization}
\author{Jonathan Weitsman}
\thanks{ }
\address{Department of Mathematics, Northeastern University, Boston, MA 02115}
\email{j.weitsman@neu.edu}
\thanks{\today}

\begin{document}
\begin{abstract}  In Weinstein's symplectic "category", objects are symplectic manifolds and morphisms are canonical relations, which may not be composable; it is therefore only morally a category.  We construct a true differential graded category which is morally dual to Weinstein's "category".   The objects in this category are prequantum systems, associated to integral symplectic manifolds.  The morphisms are a complex of differential forms twisted by a prequantum line bundle.  We also consider the cohomology category, which turns out to be an (ordinary) linear category.  In each case, Weinstein's morphisms are associated with currents dual to the forms.  For the differential graded category, these are isotropic, or if preferred Lagrangian, currents, twisted by a section of the prequantum line bundle; in the case of cohomology, these currents are supported on {\em integral} Lagrangian submanifolds, equipped with a global covariant constant section of the prequantum line bundle.  We then apply our methods to quantization, and show that for Kahler manifolds, a quantization in this category is related to holomorphic quantization.  Along the way we show that in the case of prequantum systems, the Lefschetz action of ${\mathfrak {sl}}(2,\R)$ on differential forms, studied by Brylinski, Mathieu, Guillemin, and Tseng-Yau in the symplectic case, is promoted to an action of the superalgebra ${\mathfrak {osp}}(1|2)$ on twisted differential forms, with ${\mathfrak {sl}}(2,\R)$ as the even subalgebra.   A first application of this superalgebra action is the vanishing theorem which shows the cohomology category is supported in middle dimension.

\end{abstract}

\maketitle

\tableofcontents
\section{Introduction}\label{sec:intro}

In~\cite{Weinstein}  A.~Weinstein proposed that
the symplectic geometry of quantization should be related to a symplectic category,
whose objects are symplectic manifolds $(M,\omega)$ and whose morphisms are canonical
Lagrangian relations; in other words, if $(M,\omega)$ and $(M',\omega')$ are symplectic
manifolds, a morphism $\Lambda\colon M'\Rightarrow M$ is a Lagrangian submanifold
\[
  \Lambda \subset M\times \overline{M'},
\]
where, for a symplectic manifold $M$ given by a pair $(M,\omega),$ we write $\overline{M}$ for the pair $(M,-\omega).$

Composition in this category is morally given by composition of relations. The trouble is
that such compositions yield submanifolds only given some conditions on the intersections of
the relations to be composed; so we obtain only what Guillemin and Sternberg \cite{GS} call a ``category,'' not a true
category. Nevertheless this point of view has proved very productive; see for example
the book~\cite{GS}. As an example of the intuition behind Weinstein's "symplectic creed", 
a ``point'' in this category is a morphism $\mathrm{pt}\Rightarrow M$ (where
$\mathrm{pt}$ is the symplectic manifold consisting of a single point); that is, a Lagrangian
submanifold of $M.$  This is in agreement with quantum mechanical intuition and the uncertainty principle, which
says that if the position of a particle in three dimensions (given classically by a point $q \in \R^3$) is known, its momentum (given classically
by a point $p \in T^*\R^3|_q$ in the fibre  of the cotangent bundle $T^*\R^3$ lying above $q$) cannot be localized at all; so the
quantum mechanical state corresponding to a particle located at $q \in \R^3$ is supported on the entire Lagrangian submanifold 
$T^*\R^3|_q \subset T^*\R^3$ of the symplectic manifold $T^*\R^3$ describing the phase space of the system.\footnote{Our ordering convention is with the "input" space on the right, so that a morphism $N \Rightarrow M$ is a Lagrangian submanifold of $M \times \Nbar,$ and composition is right-to-left, as for maps.  Thus a morphism ${\rm pt} \Rightarrow M$ is a Lagrangian submanifold of $M.$  The book \cite{GS} uses the opposite convention.} 

We note that Weinstein's "category" is {\em involutive}: a morphism 
$\Lambda\colon N \Rightarrow M$ given by a Lagrangian submanifold
$\Lambda \subset M \times \overline{N}$ gives an
adjoint morphism $\Lambda^T\colon M\Rightarrow N$, given by the same Lagrangian
submanifold $\Lambda$ considered as a submanifold of $N \times \overline{M}.$

The transversality troubles of Weinstein's "category" are analogous to the type of difficulty 
encountered if one attempts to define the intersection ring of a compact oriented manifold
using intersections of compact submanifolds, as in~\cite{GP}. This gives geometric
intuition about the nature of the ring, but involves many complications in perturbing submanifolds
to intersect transversally, and then showing that the resulting intersection in homology is independent of the choices made in the 
perturbation.  While these difficulties can be overcome, a simpler method of defining the
intersection ring is to use the dual theory of cohomology, where the ring structure is natural, along with Poincar\'e duality.
This avoids all the transversality problems, and shows they are irrelevant.  The price is the loss
of the geometric intuition associated with intersection theory.

In this paper we apply a similar approach to the Weinstein "category", in the case of
symplectic manifolds arising
from prequantum systems. We construct a differential graded category, whose objects are prequantum systems, and whose morphisms are given by a complex of differential forms, twisted by sections of the prequantum line bundle.  We also consider the cohomology category of this differential graded category.
In these terms morphisms in Weinstein's category appear as currents arising from Lagrangian, in fact, isotropic, submanifolds, twisted by a section of the prequantum line bundle; for the cohomology, we obtain currents arising from {\em integral} Lagrangian submanifolds, equipped with global covariant constant sections of the prequantum line bundle.  These are the submanifolds that appear naturally in quantization.
These categories provide a natural home where Lagrangian, and in fact isotropic, submanifolds
can live, even if they are topologically trivial.  We would hope this could have applications to the
study of Lagrangian submanifolds.   One angle to study would be a type of Blattner-Kostant-Sternberg pairing for integral Lagrangian currents, which our methods show (for the dual category) has cohomological content. This may give an approach to studying Lagrangian intersections --- perhaps, for example, proving nondisplaceability theorems for integral Lagrangian submanifolds of compact prequantum systems.  We study a different application, to Geometric Quantization in the Kahler case, and show that our categories give a bridge between the holomorphic world and that of Lagrangian submanifolds.

\subsection{The dual differential graded category}

Let $(M^{2m},\omega)$ be a compact connected symplectic manifold of dimension $2m,$ equipped with a hermitian line
bundle $\cL$ with unitary connection $\nabla$ of curvature $\omega$. We write $\langle\cdot,\cdot\rangle_\cL$ for the hermitian structure
on $\cL$ and denote by $c$ the resulting contraction $c: \Omega^*(M,\cL) \times \Omega^*(M,\cL) \to \Omega^*(M).$   We denote by $\hc : \Omega^*(M,\cL^*) \otimes \Omega^*(M,\cL) \to \Omega^*(M)$
the natural contraction.

A quadruple
$(M,\omega,\cL,\nabla)$ is called a \emph{prequantum system}; note that the class $[\omega] \in H^2(M,\R)$ is necessarily
integral.\footnote{By this we mean that $[\omega] \in H^2(M,2\pi\Z)$ so that $c_1(\cL) = [\omega/2\pi].$}  We consider the complex $\Omega^*(M,\cL)$ of $\cL$-valued differential forms. The
connection $\nabla$ induces an operator
\[
  \nabla \colon \Omega^*(M,\cL) \longrightarrow \Omega^{*+1}(M,\cL)
\]
with square
\[
  \nabla^2 =- i \omega\wedge. 
\]

We write

$$L := \omega \wedge$$

\noindent for the operator on $\Omega^*(M,\cL)$ given by wedge product with the symplectic form.

Consider now the subspace
\[
  \Gamma^*(M,\cL) \subset \Omega^*(M,\cL)
\]
given by
\[
  \Gamma^*(M,\cL) = \ker L.
\]
This complex is zero in degrees less than $m$ (and also of course greater than $2m$); see Proposition \ref{prop:1.1}. On
$\Gamma^*(M,\cL)$, $\nabla^2=0$, so the pair
\[
  \bigl(\Gamma^*(M,\cL),\,\nabla\bigr)
\]
is a complex, with cohomology
\[
  D^*(M,\cL) = H^*\bigl(\Gamma^*(M,\cL),\nabla\bigr).
\]

Our dual differential graded category will be built from the complex $\Gamma^*.$  We also consider the cohomology category associated to $D^*.$  We will prove the following vanishing theorem:

\begin{reptheorem}{Lagrangian}
The cohomology $D^*(M,\cL)$ is concentrated in middle dimension $m.$ \end{reptheorem}

To see the relation of $\Gamma^*$ and $D^*$ to the Weinstein "category",  we remark that a compact oriented submanifold $C\subset M$ and a section $s\in\Gamma(\cL|_C)$ give a de Rham current $\xi_{(C,s)}$ dual
 to $\Omega^*(M,\cL)$
whose value on an element 
$\eta \in \Omega^*(M,\cL)$ is given by

\[
 \xi_{(C,s)}(\eta)=  \int_C c (s, \eta) 
\]
for all $\eta\in\Omega^*(M,\cL).$  Such a current is in the kernel of the dual of $L$ if and only if 
\[
  \int_C \omega\wedge c (s, \eta)= 0
\]
for all $\eta\in\Omega^*(M,\cL);$
in other words, if $C$ is isotropic, and, in dimension $m$,
Lagrangian, on the support of $s.$ Furthermore we will show:

\begin{repprop}{prop:1.2}
  Suppose $C \subset M$ is connected.  A nonzero current $\xi_{(C,s)}$ is closed (with respect to the dual of $\nabla$) if and only if
\begin{itemize}
  \item $C$ is isotropic,
  \item $\partial C = \emptyset$,
  \item $\nabla s = 0$.
\end{itemize}
\end{repprop}

Submanifolds satisfying these conditions are called \emph{integral isotropic submanifolds}, and play an important
role in geometric quantization, where they arise as Bohr--Sommerfeld leaves of Lagrangian
foliations, and in that case give  a basis for the quantization of $M$ in the corresponding real polarization.  The relation of these currents to the differential graded category $\Gamma^*$ and its cohomology $D^*$ is the main intuition behind this paper.\footnote{A linear category with finite dimensional morphism spaces gives rise to a structure called a {\em co-category} obtained from the same object set and the dual morphism spaces.  Since the categories $\Gamma^*, D^*$ have infinite dimensional morphism spaces, the (linear) Weinstein "category" is only morally a co-category.}

The advantage of the dual differential graded category is that composition of elements of $\Gamma^*(M,\cL)$ is
well defined, with no transversality condition.  Given a prequantum system $(M,\omega,\cL,\nabla), $ write $\Mbar$ for the prequantum system $(M,-\omega,\cL^*,\nabla^+),$ where $\nabla^+$ is the connection on the dual bundle $\cL^*$ arising from the connection $\nabla$ and the hermitian structure on $\cL.$ Suppose we are given three prequantum systems arising
from symplectic manifolds $M, N, P,$ with prequantum line bundles $\cL_M, \cL_N, \cL_P,$ and
elements\footnote{We abuse notation by considering the bundles $\cL_M, \cL_N, \cL_P$ as bundles on products of the spaces $M, N, P$ via pullback.  We will continue to do this to avoid encumbering the notation.  We hope it is clear from the context which pullbacks are intended.}
\begin{align*}
  \sigma &\in \Gamma^*(M\times\Nbar,\, \cL_M\otimes \cL_N^*),\\
  \tau   &\in \Gamma^*(N\times\Pbar,\, \cL_N\otimes \cL_P^*).
\end{align*}
We then define the composition of $\sigma$ and $\tau$ 
\[
  \sigma\circ\tau \in \Omega^*(M\times\Pbar,\, \cL_M\otimes \cL_P^*)
\]
as
\[
  \sigma\circ\tau \;=\; k_*\,\hc_N(i^*\sigma\wedge j^*\tau),
\]
where
\[
  i\colon M\times N\times P \to M\times N,\qquad
  j\colon M\times N\times P \to N\times P
\]
are the projections, $k_*$ is fibre integration along the map

\[
  k \colon M\times N\times P \to M\times P,
\]
\noindent and we write $\hc_N$ for the contraction
\[
  \cL_M\otimes \cL_N^*\otimes \cL_N\otimes \cL_P^* \longrightarrow \cL_M\otimes \cL_P^*
\]
\noindent arising from the natural contraction $\hc_N : \cL_N^*\otimes \cL_N \to \C, $ and similarly for differential forms.

To show that the composition of two elements of $\Gamma^*$ actually yields an element of $\Gamma^*$, we prove the following results about
the operation $\circ.$

\begin{reptheorem}{leibniz}
Suppose we are given three prequantum systems arising
from symplectic manifolds $M, N, P,$ with prequantum line bundles $\cL_M, \cL_N, \cL_P.$
Let $\sigma\in\Gamma^*(M\times\Nbar,\, \cL_M\otimes \cL_N^*)$ and
$\zeta\in\Gamma^*(N\times\Pbar,\, \cL_N\otimes \cL_P^*)$. Then\footnote{In this paper we write $|\sigma|$ for the degree of a differential form $\sigma$ (with values in any sheaf).}
\[
  \nabla(\sigma\circ\zeta) \;=\; \nabla\sigma\circ\zeta \;+\; (-1)^{|\sigma|}\sigma\circ\nabla\zeta.
\]
\end{reptheorem}
As a corollary we have

\begin{repcorollary}{Lcomp}
Suppose we are given three prequantum systems arising
from symplectic manifolds $M, N, P,$ with prequantum line bundles $\cL_M, \cL_N, \cL_P.$
Let $\sigma\in\Gamma^*(M\times\Nbar,\, \cL_M\otimes \cL_N^*)$ and
$\zeta\in\Gamma^*(N\times\Pbar,\, \cL_N\otimes \cL_P^*)$. Then
\[
 L(\sigma\circ\zeta) \;=\; L\sigma\circ\zeta \;+ \sigma\circ L \zeta.
\]
\end{repcorollary}

Thus

\begin{repcorollary}{compos} 

Suppose we are given three prequantum systems arising
from symplectic manifolds $M, N, P,$ with prequantum line bundles $\cL_M, \cL_N, \cL_P,$ and
elements
\begin{align*}
  \sigma &\in \Gamma^*(M\times\Nbar,\, \cL_M\otimes \cL_N^*),\\
  \tau   &\in \Gamma^*(N\times\Pbar,\, \cL_N\otimes \cL_P^*).
\end{align*}

Then the composition 

\[
  \sigma\circ\tau \in \Omega^*(M\times\Pbar,\, \cL_M\otimes \cL_P^*)
\]

lies in $ \Gamma^*(M\times\Pbar,\, \cL_M\otimes \cL_P^*).$\end{repcorollary}

The following result shows the composition map $\circ$ is associative:

\begin{repprop}{assoc}
Suppose we are given four prequantum systems corresponding to symplectic manifolds $M, N, P,Q, $ with prequantum line bundles $\cL_M, \cL_N, \cL_P,\cL_Q.$  Then for any
\(\sigma \in \Gamma^*(M\times\Nbar,\, \cL_M\otimes \cL_N^*)\),
\(\tau   \in \Gamma^*(N\times\Pbar,\, \cL_N\otimes \cL_P^*)\), and
\(\rho   \in \Gamma^*(P\times\Qbar,\, \cL_P\otimes \cL_Q^*)\);
we have
\[
  (\sigma\circ\tau)\circ\rho \;=\; \sigma\circ(\tau\circ\rho).
\]
\end{repprop}

Unlike the Weinstein "category", where the diagonal Lagrangian in $M \times \Mbar$ is the identity morphism, our category does not have such a morphism.  To get around this, simply adjoin to the morphisms arising from $\Gamma^*$ an identity morphism $1_M$ for every symplectic manifold; in fact, one may as well adjoin a morphism $S_M$ for every symplectomorphism $S : M \to M$ preserving the hermitian line bundle $\cL$ and the connection 
$\nabla.$\footnote{Such symplectomorphisms are sometimes called {\em quantomorphisms}.}

We also have a form of involutivity:  For any two  prequantum systems arising
from symplectic manifolds $M, N $ with prequantum line bundles $\cL_M, \cL_N, $\ and given $\sigma\in\Gamma^*(M \times\Nbar,\, \cL_M\otimes \cL_N^*)$, we write
\[
  \sigma^T \in \Gamma^*(N\times\Mbar,\, \cL_N\otimes \cL_M^*),
\]
for the form arising from $\sigma$ via the map $M \times N \to N \times M$ exchanging $M$ and $N$ and the hermitian structures on $\cL_M$ and $\cL_N.$  It is then clear that $(\sigma^T)^T = \sigma$ for any  $\sigma\in\Gamma^*(M \times\Nbar,\, \cL_M\otimes \cL_N^*).$  This involutivity is graded-contravariant with respect to the operation $\circ:$  If $ \sigma \in \Gamma^*(M\times\Nbar,\, \cL_M\otimes \cL_N^*),
\tau   \in \Gamma^*(N\times\Pbar,\, \cL_N\otimes \cL_P^*)$ as above, then

$$(\sigma \circ \tau)^T = (-1)^{|\sigma||\tau|} \tau^T \circ \sigma^T.$$

We summarize these results in the following

\begin{reptheorem}{gammacat}
There exists a $\Z_2$-graded, involutive differential graded category $\Gamma^*$ whose objects are prequantum systems $(M,\omega,\cL,\nabla)$, and whose
morphisms $N\Rightarrow M$ between two such systems $(M,\omega_M,\cL_M,\nabla_M)$ and $(N,\omega_N,\cL_N,\nabla_N)$ are given by the complex
\[
  \Gamma^*(N,M) := \Gamma^*(M\times\Nbar,\, \cL_M\otimes \cL_N^*)
\]
with the composition law~$\circ.$  
\end{reptheorem}

\begin{Remark} The fact that $\Gamma^*$ is only a $\Z_2$-graded category is due to the fact that the composition operation $\circ$ is not of degree zero:   If $ \sigma \in \Gamma^a(M\times\Nbar,\, \cL_M\otimes \cL_N^*),
\tau   \in \Gamma^b(N\times\Pbar,\, \cL_N\otimes \cL_P^*),$ then $\sigma \circ \tau \in \Gamma^{a+b -2n} (M\times\Pbar,\, \cL_M\otimes \cL_P^*).$ It is possible to regrade the complex $\Gamma^*$ and add a sign to the composition operation $\circ$ to obtain a $\Z$-graded differential graded category.  See Theorem \ref{gammacat} [Second version] of Section \ref{prom}.
\end{Remark}
Note that if ${\rm dim~} M = 2m , $ ${\rm dim} ~N = 2n,$ and ${\rm dim}~ P = 2p, $ compositions of elements in $\Gamma^{m + n} (M,N)$ with elements
in $\Gamma^{n+p}(N,P)$ give elements of $\Gamma^{m + p}(M,P)$; so that the category
$\Gamma^*$ contains a subcategory with morphisms
\[
  \Gamma_{Lag}(N,M) := \Gamma^{m+n} (M^{2m} \times\Nbar^{2n},\, \cL_M\otimes \cL_N^*)
\]
corresponding to (duals of) Lagrangian submanifolds.   The involution of the differential graded category, which is graded contravariant, is contravariant in the linear sense if the middle dimensional forms are even, that is, if the dimension of the manifold is a multiple of four.

\begin{Remark}\labell{meta} It is natural to ask whether this anomaly in the involution is related to the metaplectic correction that arises in geometric quantization.\end{Remark}

\begin{reptheorem}{laggammacat}
There exists a linear category $\Gamma_{Lag}$ whose objects are prequantum systems $(M^{2m},\omega,\cL,\nabla)$, and whose
morphisms $N^{2n}\Rightarrow  M^{2m} $ between two such systems $(M^{2m},\omega_M,\cL_M,\nabla_M)$ and $(N^{2n},\omega_N,\cL_N,\nabla_N)$ are given by the vector space
$$\Gamma_{Lag} (N^{2n},M^{2m}): =\Gamma^{m+n}(M\times\Nbar,\, \cL_M\otimes \cL_N^*)
$$
with the composition law~$\circ.$  Restricted to manifolds whose dimension is a multiple of $4,$ this category is involutive.
\end{reptheorem}

Theorem \ref{leibniz} shows that this structure descends to the cohomology $D^*;$ recall that by Theorem \ref{Lagrangian}, this is supported in middle dimension. 

\begin{reptheorem}{dcat}
There exists a linear category $D^*$ whose objects are prequantum systems $(M^{2m},\omega,\cL,\nabla)$, and whose
morphisms $N\Rightarrow M$ between two such systems $(M^{2m},\omega_M,\cL_M,\nabla_M)$ and $(N^{2n},\omega_N,\cL_N,\nabla_N)$ are given by the cohomology
\[
   D_{Lag}(N,M): = D^{m+n}(M\times\Nbar,\, \cL_M\otimes \cL_N^*)
\]
with the composition law~$\circ.$  Restricted to manifolds whose dimension is a multiple of $4,$ this category is involutive.
\end{reptheorem}

In Section~\ref{reltow} we will show that morally, the composition law in these
dual categories has a close resemblance, using the currents $\xi_{(C,s)}$, to Weinstein's composition
law. This justifies the use of the term "dual category".

As a first application of our methods, we use our construction to study Geometric Quantization, and to
investigate invariance of polarization.

\subsection{Application to geometric quantization}

Geometric quantization would assign to a prequantum system $(M,\omega,\cL,\nabla)$ a vector
space $Q(M)$ with ``good'' properties, not exactly specified but understood in many
examples. In these examples the quantization $Q(M)$ is obtained from the space of sections $\Gamma(\cL)$
by ``cutting out'' a subspace, which for $M$ compact is expected to be finite dimensional.
To do so, the quadruple $(M,\omega,\cL,\nabla)$ is supplemented by a polarization; two examples
are given by a Kahler structure on $(M,\omega)$ and a holomorphic structure on $(\cL,\nabla)$,
or, alternatively, by a Lagrangian foliation given by a map
\[
  \pi\colon M^{2m} \to B^m
\]
for some $B$.  In the Kahler case, we take $Q(M) = H^0(M,\cL),$ the space of holomorphic sections of $\cL;$ in the Lagrangian case,
$Q(M)$ consists of distributional sections supported on the integral leaves $\pi^{-1}(b)$
of the foliation, called in this context Bohr-Sommerfeld leaves, corresponding to points $b\in B$ lying in a finite subset $B_{BS}\subset B$, and generated by the covariant constant section of
$\cL|_{\pi^{-1}(b)}$. Remarkably, in many cases the quantization turns out to be independent
of the polarization, although the relation between Lagrangian foliations and complex
structures is obscure, to say the least.

We have seen that our $D^m(M^{2m},\cL)$ is dual to integral Lagrangian submanifolds. The
following theorem relates this space to Kahler quantization.

\begin{reptheorem}{kahlergq}
 
Let $(M,\omega)$ be a compact Kahler manifold, and let $(\mathcal{L},\nabla)$ be a
hermitian holomorphic line bundle with Chern connection of curvature $\omega.$ Let $k$ be sufficiently large
and write $\nabla_k$ for the connection on $\cL^k$ induced by $\nabla.$  Define
\[
  Q(M,\cL^k ) \;=\; \ker(\nabla_k)\cap \ker(\nabla_k^*) \;\subset\; \Gamma^*(M,\mathcal{L}^k).
\]

\noindent where we continue to write $\nabla_k$ for the operator induced by the connection $\nabla_k$ on $\Gamma^*(M,\mathcal{L}^k)$ and $\nabla_k^*$ for its adjoint.
Then
\[
  Q(M,\cL^k) \;\simeq\; H^0(M,\, \mathcal{L}^k \otimes K_M  ),
\]
where $K_M$ is the canonical bundle of $M$.
 
\end{reptheorem}

Along the way we find that the twisting of the de Rham operator on a symplectic  manifold arising from a prequantum system gives rise to an extension of the
Lefschetz $\sla(2,\R)$ action on differential forms on a symplectic manifold to that of the superalgebra ${\mathfrak {osp}}(1|2)$ on twisted differential forms. As a first application we have a vanishing theorem (Theorem \ref{Lagrangian}) for the cohomology $D^*(M,\cL).$  We hope that this twisted version
of symplectic Hodge theory \cite{B,G,M,TY} 
will give further insight into geometric quantization, beyond the Kahler case.  

The rest of the paper is structured as follows.  In Section \ref{osp}, we study the ${\mathfrak {sl}}(2,\R)$ action on twisted differential forms, and its extension in the prequantum case to an action of ${\mathfrak {osp}}(1|2).$  We then review in Section \ref{sec:weinstein} the Weinstein "category" and also the case of integral Lagrangian morphisms.  The next Section, Section \ref{sec:complex}, is the core of the paper, where we define the dual categories $\Gamma^*$ and $D^*$ and study their properties.  Finally in Section \ref{gq} we give an application of our methods to Geometric Quantization, and show that in the Kahler case, our category gives a bridge between the Lagrangian and holomorphic worlds.  We conclude with some conjectures and questions that arise from this application.

\section{The Lefschetz ${\mathfrak {sl}}(2,\R)$ and the prequantum ${\mathfrak {osp}}(1|2)$}\labell{osp}

Let $(M^{2m},\omega)$ be a compact connected symplectic manifold of dimension $2m.$
  The ring of differential forms $\Omega^*(M)$ is equipped with a symplectic Hodge star type operator we denote by $\star_s$ \cite{B,G,M,TY},\footnote{In the following we use the conventions of Tseng-Yau \cite{TY}.} satisfying

\begin{equation}\labell{ssq}(\star_s)^2 = 1.\end{equation} 

On $\Omega^*(M)$ we also have the operator $$L =  \omega \wedge$$ and the operator $$\Lambda = \Pi \lrcorner$$ where  $\Pi$ is the Poisson bivector field associated to $\omega,$ and $\lrcorner$ is interior product.  We note from \cite{TY} that $\Lambda$ and $L$ are related by

\begin{equation}\labell{lam}\Lambda = \star_s L \star_s.\end{equation} 

We then have

$$[\Lambda,L]= H,$$

\noindent where $H$ is the (shifted, negative) grading operator defined by

\begin{equation} \labell{hdef} H|_{\Omega^k(M)} = (m-k) 1.\end{equation}

Since $H$ is a (shifted, negative) grading operator, and $L, \Lambda$ shift degrees by $\pm 2,$ we have

$$[H,L] = -2L$$

\noindent and

$$[H,\Lambda] = 2 \Lambda$$

\noindent Thus the operators $L,\Lambda,H$ generate a representation of ${\mathfrak {sl}}(2,\R).$  This is  the symplectic construction of the Lefschetz algebra \cite{B,G,M,TY}. In the Kahler case, one may write $\Lambda = L^*$, where $L^*$ denotes the metric adjoint of $L.$  See e.g. \cite{gh,Wells}.

Suppose now that $M$ is equipped with a hermitian line bundle $\cL$ with connection $\nabla$ of curvature $\omega.$  Then the operators $\star_s,L,\Lambda,H$ can be defined in the same way also on the differential forms $\Omega^*(M,\cL),$ and satisfy the same relations.  We abuse notation slightly by continuing to denote the resulting operators by $\star_s,L, \Lambda, H.$

In addition, in the prequantum case, we have the operator $\nabla: \Omega^*(M,\cL) \to \Omega^{*+1}(M,\cL)$ arising from the connection, which satisfies

$$\{\nabla,\nabla\} = - 2 iL$$

\noindent where we have denoted by $\{\cdot, \cdot\}$ the graded bracket  $\{A, B\} = AB + BA.$

Let us write, following \cite{TY} 

\begin{equation}\labell{nablalam} \nabla^\Lambda = [\nabla,\Lambda].\end{equation}

We will show that

\begin{lemma}\labell{nablalambda}

The operator $\nabla^\Lambda$ satisfies

$$ \nabla^\Lambda|_{\Omega^k(M,\cL) }= (-1)^{k+1} \star_s \nabla \star_s.$$

\end{lemma}

We then have by (\ref{ssq}) and (\ref{lam}) 

$$\{\nabla^\Lambda,\nabla^\Lambda \} =  2 i\Lambda.$$

We now compute the remaining brackets in the algebra generated by the five operators $L,\Lambda, H, \nabla, \nabla^\Lambda.$  To begin with we have

$$[\nabla,L] = 0$$

$$[\nabla^\Lambda,\Lambda]=0,$$

\noindent immediate since $iL = - \nabla^2$ and $i\Lambda =  (\nabla^\Lambda)^2$ .

Lemma \ref{nablalambda}, along with (\ref{lam}) and (\ref{nablalam}) gives

$$[\nabla^\Lambda,L] = \nabla.$$

Since $H$ is a (negative, shifted) grading operator, we also have

$$[H,\nabla] = - \nabla$$

\noindent and

$$[H,\nabla^\Lambda] = \nabla^\Lambda.$$

Finally we compute

$$\{\nabla,\nabla^\Lambda\} = \{\nabla, [\nabla,\Lambda]\} = \nabla \nabla \Lambda - \nabla \Lambda \nabla + \nabla \Lambda \nabla - \Lambda \nabla \nabla = [-iL,\Lambda] = iH.$$

Thus we have proved

\begin{lemma}\label{comms}
We have the following relations:

\begin{enumerate}
\item $ [\Lambda,L ] = H.$
\item $[H,L] = -2L.$
\item $[H,\Lambda] = 2\Lambda.$
\item $[L,\nabla] = 0.$ 
\item $[\Lambda,\nabla^\Lambda] = 0.$
\item $[ \nabla,\Lambda] = \nabla^\Lambda.$
 \item $[\nabla^\Lambda,L] = \nabla.$
\item $[H,\nabla] = - \nabla.$
\item $[H,\nabla^\Lambda] = \nabla^\Lambda.$
\item $\{\nabla,\nabla^\Lambda \} = iH.$
\item $\{\nabla,\nabla\} = - 2 i L.$
\item $\{\nabla^\Lambda,\nabla^\Lambda\} = 2 i \Lambda.$
\end{enumerate}
 \end{lemma}
 
 This can be rephrased in terms of known superalgebras as follows:
 
\begin{theorem}\labell{osp12}

Let $(M,\omega,\cL,\nabla)$ be a prequantum system. The operators $\nabla, \nabla^\Lambda, \Lambda, L, H$ give a representation of the superalgebra ${\mathfrak {osp}}(1|2)$ on the complex of twisted differential forms $\Omega^*(M,\cL).$
\end{theorem}
\begin{proof} By items (10), (11), and (12) of Lemma \ref{comms}, our algebra is generated by two odd elements $\nabla, \nabla^\Lambda$ and one even element $H.$  By items (8), (9), and (10) of Lemma \ref{comms}, these generators  satisfy (up to normalization) the relations of the generators of ${\mathfrak {osp}}(1|2)$ as in \cite{fss}, p. 75 or \cite{kac}. \end{proof}

It remains to prove Lemma \ref{nablalambda}.

\begin{proof}   First we note that for operators $d, \Lambda, \star_s$ on ordinary differential forms, which can be thought of as the differential forms with values in the trivialized bundle, with connection given by $d,$  Tseng-Yau (\cite{TY}, equation (2.4)) show 

\begin{equation} \labell{ty1}[d,\Lambda]|_{\Omega^k(M) }= (-1)^{k+1} \star_s  d \star_s. \end{equation}

\noindent For any function $f\in C^\infty(M),$ define $d_f : \Omega^*(M) \to \Omega^{*+1}(M)$ by

$$d_f \xi = e^{-f} d (e^f \xi)$$

\noindent for $\xi \in \Omega^*(M).$ Then, since multiplication operators commute with $\Lambda,$ we also have  

\begin{equation} \labell{tyn}[d_f, \Lambda]|_{\Omega^k(M) }= (-1)^{k+1}  \star_s d_f \star_s.\end{equation}

\noindent But $$d_f = d + df \wedge,$$ 

\noindent so that, subtracting (\ref{ty1}) from (\ref{tyn}), we have

\begin{equation} \labell{dfe}[ df\wedge,\Lambda ]|_{\Omega^k(M) }= (-1)^{k+1}  \star_s  (df\wedge)\star_s \end{equation}

Now, for any one form $ a \in \Omega^1(M),$ consider the operator 

$$a \wedge: \Omega^*(M) \to \Omega^{*+1}(M).$$  

This operator, and its bracket with $\Lambda,$ are determined by their pointwise values.  Since any element of the cotangent bundle at a point of $M$ can be presented as the value of $df$ at that point for some $f \in C^\infty(M),$
equation (\ref{dfe}) implies

\begin{equation} \labell{tynn} [a \wedge,\Lambda]|_{\Omega^k(M) }= (-1)^{k+1}  \star_s (a \wedge)\star_s \end{equation}

\noindent for any one form $a \in \Omega^1(M).$

But locally, in any ball in which $\cL$ is trivial, and any trivialization of $\cL$ on that ball, any connection on $\cL$ is of the form $d + a \wedge$ for some one form $a$ on the ball.  Thus, combined with (\ref{ty1}), equation (\ref{tynn}) gives the Lemma.\end{proof} 

For a first application of the ${\mathfrak {osp}}(1|2)$ action to a vanishing theorem in our category, see Theorem \ref{Lagrangian}.  We would hope that there are further applications of the representation theory of ${\mathfrak {osp}}(1|2)$ to quantization, just as the Lefschetz ${\mathfrak {sl}}(2,\R)$ has been useful in Kahler geometry.

\section{The Weinstein ``category''}\label{sec:weinstein}

Weinstein's symplectic creed motivates an attempt to define a category whose objects are
symplectic manifolds and whose morphisms are canonical Lagrangian relations~--- a morphism
$(N,\omega_N)\Rightarrow(M,\omega_M) ,$
where $M, N$ are symplectic manifolds, is a Lagrangian submanifold of $M\times \Nbar$,
where $\Nbar = (N,-\omega_N)$. The composition is the composition of relations, and the identity morphism is associated with the diagonal in $M \times \Mbar,$ which is a Lagrangian submanifold.
The composition of relations is always well defined, but the composition of two relations arising from
Lagrangian submanifolds yields a relation arising from a Lagrangian submanifold only given a composability
condition.  Thus composition is not really well defined, and we obtain only what Guillemin-Sternberg  \cite{GS} call a ``category.''

The situation is analogous to the definition of the homology intersection ring of
a compact manifold, where submanifolds representing homology classes must be deformed to be transversal in order to get a submanifold as their intersection.
The most convenient way to avoid this problem is via cohomology, and the methods of the
present paper are in some sense a symplectic geometric analog of that construction.
However, just as in algebraic topology, the geometric intuition comes from intersections of submanifolds,
despite the transversality problem.  We first give a review of Weinstein's construction and
then study it in the integral case, where the symplectic manifolds arise from prequantum systems and the canonical relations are {\em integral} Lagrangian submanifolds, that is, Lagrangian submanifolds on which the prequantum line bundle is trivial as a line bundle with connection, equipped with a global covariant constant section of this line bundle. These two cases will turn out to give the currents dual to the forms in the categories $\Gamma^*$ and its cohomology $D^*,$ which will appear in the next section.\footnote{Strictly speaking, the categories $\Gamma^*$ and $D^*$ are (morally) dual to linearized versions of the Weinstein "category".  In the case of $\Gamma^*$ the objects in this "category" are built from symplectic manifolds, and the morphisms from oriented isotropic submanifolds, equipped with arbitrary sections of the restriction of the prequantum line bundle.  In the case of $D^*$, the morphisms are built from oriented integral Lagrangian submanifolds, equipped with global covariant constant sections of the restriction of the prequantum line bundle.  See Section \ref{linwein}.}

Since this section is intended as motivation, the presentation is informal.

\subsection{Weinstein's construction}

Weinstein builds a ``category,'' whose objects are symplectic manifolds and whose
morphisms are Lagrangian submanifolds, where composition is composition of relations. Let
$M$, $N$, $P$ be symplectic manifolds, and let
\[
  \Lambda \subset M\times\Nbar, \qquad
  \Lambda' \subset N\times\Pbar.
\]

\noindent be Lagrangian submanifolds.

The composition $\Lambda\circ\Lambda' \subset M\times\Pbar$ is given by
\[
  \Lambda\circ\Lambda' \;=\; \bigl\{(m,p)\in M\times\Pbar :
  \exists\, n\text{ with } (m,n)\in\Lambda,\,(n,p)\in\Lambda'\bigr\}.
\]

The trouble is that $\Lambda\circ\Lambda'$ may not be a submanifold of $M\times\Pbar$.
To avoid this problem, we impose a geometric condition on $\Lambda, \Lambda'.$ To do this, note that the diagonal $\Delta_{\Nbar} \subset \Nbar \times N$ given by
\[
  \Delta_{\Nbar} = \{(n,n):n\in N\}
\]
is a Lagrangian submanifold of $\Nbar \times N$; and then $ \Lambda\circ\Lambda'$ is the image of 
$ (\Lambda\times\Lambda') \cap (M\times\Delta_{\Nbar} \times\Pbar) \subset M \times N \times P$ in $M \times P$
under the projection.  Then \cite[Theorem 4.1.1]{GS}:

\begin{proposition}\label{prop:3.1}
$\Lambda\circ\Lambda'$ is an immersed Lagrangian submanifold of $M\times\Pbar$ provided that
$\Lambda\times\Lambda'$ meets $M\times \Delta_{\Nbar} \times \Pbar$ transversally.
\end{proposition}

For the composition $\Lambda \circ \Lambda'$ to be a Lagrangian submanifold of $M \times \Pbar,$ we must impose an additional condition.

\begin{definition}\label{def:3.2}
\cite{GS}.
Let $\Lambda\subset M\times\Nbar$, $\Lambda'\subset N\times\Pbar$ be Lagrangian.
The Lagrangian submanifolds $\Lambda,\Lambda'$ are \emph{composable} if $\Lambda\times\Lambda'$ meets $M\times \Delta_{\Nbar} \times \Pbar$ transversally, and, in addition, for any $m\in M$, $p\in P$, there exists at
most one point $n\in N$ with $(m,n)\in\Lambda$, $(n,p)\in\Lambda'$.\footnote{The definition in \cite{GS} allows a more general composability condition, corresponding to clean intersection of $\Lambda\times\Lambda'$ with $M\times \Delta_{\Nbar} \times \Pbar,$ and with a connectedness condition replacing our unique point condition.}
\end{definition}

\begin{proposition}\label{prop:3.3}
\cite[Theorem 4.2.3]{GS}.
Suppose $\Lambda$, $\Lambda'$ are composable. Then $\Lambda\circ\Lambda'$ is a Lagrangian
submanifold of $M\times\Pbar$.
\end{proposition}

\begin{Remark}
The constructions in this section generalize to isotropic submanifolds (with transversality replaced in the definition of composability by clean intersection, as in \cite{GS}).
\end{Remark}

Note that in Weinstein's construction, morphisms from a point ${\rm pt},$ considered as a symplectic manifold, to a symplectic manifold $M,$ are Lagrangian submanifolds of $M.$  

\subsection{Prequantum systems and integral Lagrangian morphisms}

Now let $(M, \omega_M, \cL_M, \nabla_M)$ and $(N, \omega_N, \cL_N, \nabla_N)$ be
prequantum systems:  That is, $(M,\omega_M)$ is a symplectic manifold equipped with a hermitian line bundle $\cL_M$ with hermitian structure $\langle\cdot,\cdot\rangle_{\cL_M},$ and unitary connection $\nabla_M$ of curvature $\omega_M;$ and likewise for $N.$

Recall that a Lagrangian submanifold of a prequantum system $(M, \omega_M, \cL_M, \nabla_M)$ is {\it integral} if
$(\cL_M,\nabla_M)|_\Lambda$ is trivial as a bundle with connection, so that $\cL_M |_\Lambda$ has a
global covariant constant section $s\in \Gamma(\cL_M|_\Lambda)$. (If $\Lambda$ is connected,
$s$ is unique up to a constant multiple.)
\begin{definition}\label{def:3.5}
An \emph{integral Lagrangian morphism} $\Lambda\colon N \Rightarrow M$ of prequantum systems\newline  
$(M, \omega_M, \cL_M, \nabla_M),$  $(N, \omega_N, \cL_N, \nabla_N)$  is an
integral Lagrangian submanifold
\[
  \Lambda \subset M\times \Nbar
\]
equipped with a covariant constant section of the bundle $(\cL_M \otimes \cL_N^*)|_\Lambda.$

\end{definition}

Note that an integral Lagrangian morphism from a point ${\rm pt},$ considered as a prequantum system, to a prequantum system $M,$ is an integral Lagrangian submanifold of $M,$ equipped with a global covariant constant section of the prequantum line bundle.

We may similarly define integral isotropic morphisms.

We now show that the composition of integral Lagrangian morphisms, when they are composable, is an integral Lagrangian morphism.

\begin{proposition}\label{prop:3.7}  Let $(M, \omega_M, \cL_M, \nabla_M), $ $(N, \omega_N, \cL_N, \nabla_N)$ and $(P, \omega_P, \cL_P, \nabla_P)$ be prequantum systems.  Suppose we are given integral Lagrangian submanifolds  
 $\Lambda \subset M\times\Nbar, 
  \Lambda' \subset N\times\Pbar,$ which are composable as Lagrangian morphisms. Then $\Lambda\circ\Lambda'$ is an
integral Lagrangian submanifold of $M\times\Pbar$.

\end{proposition}

\begin{proof}
We know by Proposition~\ref{prop:3.3} that $\Lambda\circ\Lambda'$ is a Lagrangian
submanifold of $M\times\Pbar$. It remains to check that this submanifold is integral,
that is, that the bundle $(\cL_M \otimes \cL_P^*)|_{\Lambda \circ \Lambda'}$ is trivial as 
a bundle with connection.

Since $\Lambda$ and $\Lambda'$ are integral Lagrangian submanifolds of $M \times \Nbar,$ $N \times \Pbar,$ respectively, we know that

$$(\cL_M \otimes \cL_N^* )|_\Lambda \simeq \C$$

\noindent and

$$(\cL_N \otimes \cL_P^* )|_{\Lambda'} \simeq \C$$

\noindent are trivial as line bundles with connection.

Hence 
 
$$((\cL_M \otimes \cL_N^*) \otimes (\cL_N \otimes \cL_P^*))|_ {\Lambda \times \Lambda'} \simeq \C$$

\noindent is also trivial as a line bundle with connection.\footnote{As we noted in the introduction, we have abused notation by considering the bundles $\cL_M, \cL_N, \cL_P$ as bundles on products of the spaces $M, N, P$ via pullback.  We will continue to do this to avoid encumbering the notation.  We hope it is clear from the context which pullbacks are intended.}

Restricting to the diagonal  $\Delta_{\Nbar},$ we see that the bundle 

$$((\cL_M \otimes \cL_N^*) \otimes (\cL_N \otimes \cL_P^*))|_ {(\Lambda \times \Lambda') \cap (M \times \Delta_{\Nbar} \times \Pbar )} \simeq \C$$

\noindent is also trivial as a bundle with connection.

But on the diagonal, $\cL_N^* \otimes \cL_N \simeq \C$.\footnote{Here our abuse of notation is a bit problematic;  we have a copy of $\cL_N^*$ pulled back to $(\Lambda \times \Lambda') \cap (M \times \Delta_{\Nbar} \times \Pbar) $   from $\Lambda,$ and a copy of $\cL_N$ pulled back to $(\Lambda \times \Lambda') \cap (M \times \Delta_{\Nbar} \times \Pbar)$ from $\Lambda'.$  Restricted to $(\Lambda \times \Lambda' )\cap (M \times \Delta_{\Nbar} \times \Pbar) $ the tensor product of these two bundles is canonically the trivial bundle.}

Hence 

$$(\cL_M \otimes \cL_P^* )|_{\Lambda \circ \Lambda'} $$

is trivial as a line bundle with connection.

\end{proof}

We have an explicit formula for the covariant constant section on $(\cL_M \otimes \cL_P^* )|_{\Lambda \circ \Lambda'},$ arising from the sections on $\Lambda$ and $\Lambda',$ which follows from the isomorphism of Proposition \ref{prop:3.7}:

\begin{proposition}\label{prop:3.8}  Suppose $\Lambda, \Lambda'$ are composable integral Lagrangian morphisms as in the statement of Proposition \ref{prop:3.7}.
Suppose $s\in\Gamma((\cL_M\otimes \cL_N^*)|_\Lambda )$, $s'\in\Gamma((\cL_N\otimes \cL_P^*)|_{\Lambda'})$
are covariant constant.

Let
\begin{equation}\labell{compeq}
  (s\circ s')(m,p) \;=\; \hc_N( s(m,n) s'(n,p)),
\end{equation}
\noindent where $n\in N$ is the unique point with $(m,n)\in\Lambda$, $(n,p)\in\Lambda'$
and $\hc_N$ is the contraction $\cL_M \otimes \cL_N^* \otimes \cL_N \otimes \cL_P^* \to \cL_M  \otimes \cL_P^*$ induced by the fibrewise pairing of $\cL_N$ with its dual.

Then $s \circ s'$ is a covariant constant section of $(\cL_M \otimes \cL_P^* )|_{\Lambda \circ \Lambda'}.$
\end{proposition}

\begin{Remark}  The formula (\ref{compeq}) can also be used to define composition for the enhancement of the Weinstein "category" built from Lagrangian, or isotropic, submanifolds of symplectic manifolds coming from prequantum systems, and equipped with {\it arbitrary} sections of the prequantum line bundle.  This enhanced category is the version of the Weinstein "category" giving the currents dual to our $\Gamma^*.$\end{Remark}

Note that this "category" is involutive; given an integral Lagrangian submanifold $\Lambda \subset M \times \Nbar$ along with a covariant constant section $s \in \Gamma((\cL_M\otimes \cL_N^*)|_\Lambda)$, the Lagrangian submanifold $\Lambda,$ thought of as a Lagrangian submanifold   $\Lambda^T \subset N \times \Mbar,$ comes equipped with a covariant constant section $s^T \in \Gamma((\cL_N\otimes \cL_M^*)|_{\Lambda^T}),$ arising from the hermitian metrics  on the line bundles $\cL_M, \cL_N.$

\begin{Remark} The structure of an involutive "category" does not quite exhaust the structure present in Weinstein's construction:  A Lagrangian submanifold $\Lambda \subset M \times \Nbar$ is not only a morphism from $N$ to $M,$ but also a morphism from ${\rm pt}$ to $M \times \Nbar.$  The same will be true for our dual category.  This hints at the presence of a more involved algebraic structure.
\end{Remark}

One particular case is that of a morphism $\Lambda: \mathrm{pt}\Rightarrow M$, given
by an integral Lagrangian submanifold $\Lambda\subset M,$ equipped with a global covariant constant section
$s\in \Gamma(\cL_M|_\Lambda)$.

\begin{proposition}\label{prop:3.9}
Let $\Lambda_1, \Lambda_2 \subset M$ be integral Lagrangian submanifolds of $M$,
equipped with covariant constant sections $s_1\in\Gamma(\cL|_{\Lambda_1})$, $s_2\in\Gamma(\cL|_{\Lambda_2}).$
Suppose $\Lambda_1$ and $\Lambda_2$ intersect transversally at a single point. Then the morphisms
$\Lambda_1\colon \mathrm{pt}\Rightarrow M$, $\Lambda_2^T\colon M \Rightarrow \mathrm{pt}$
are composable,
\[
  \Lambda_2^T\circ\Lambda_1 = {\rm pt} \]
and
\[
  s_2^T \circ s_1 =  \bigl\langle s_2(m),\, s_1(m)\bigr\rangle_{\cL_M},
\]
where $\{m\}=\Lambda_1\cap\Lambda_2$.
\end{proposition}

Where these Lagrangians arise from real polarizations, that is, foliations of $M$ by Lagrangian submanifolds, this is a version of the Blattner--Kostant--Sternberg pairing.

\begin{Remark}
The constructions in this section generalize to isotropic submanifolds (with transversality replaced in the definition of composability by clean intersection, as in \cite{GS}).
\end{Remark}

\subsection{Linear Weinstein "categories"}\labell{linwein}  It is also possible to construct linear versions of the two Weinstein "categories" considered in this section.  

For the case of isotropic submanifolds, consider the category whose objects are prequantum systems, and whose morphisms are given by linear combinations of oriented isotropic submanifolds, equipped with sections of the prequantum line bundle.  In more detail, given prequantum systems $(M, \omega_M, \cL_M, \nabla_M)$ and $(N, \omega_N, \cL_N, \nabla_N),$ a morphism of the linearized Weinstein "category" is a linear combination 

\[
\sum_i a_i [\Lambda_i] s_i
\]

\noindent where for each $i,$ $a_i \in \C,$ $\Lambda_i$ is an oriented isotropic submanifold of $M \times \Nbar,$  $s_i$ is a section of $(\cL_M \otimes \cL_N^*)|_{\Lambda_i},$ and, as is usual in homology, for any oriented submanifold $ \Lambda  \subset M \times \Nbar,$ we identify $-[\Lambda]$ with the submanifold $\Lambda$ with the opposite orientation.

Likewise a linearized version of the integral Weinstein "category" can be obtained by taking the category whose objects are prequantum systems, and whose morphisms are given by linear combinations of oriented integral Lagrangian submanifolds, equipped with global covariant constant sections of the prequantum line bundle.  In more detail, given prequantum systems $(M, \omega_M, \cL_M, \nabla_M)$ and $(N, \omega_N, \cL_N, \nabla_N),$ a morphism of the linearized Weinstein "category" is a linear combination 

\[
\sum_i a_i [\Lambda_i] s_i
\]

\noindent where for each $i,$ $a_i \in \C,$ $\Lambda_i$ is an oriented integral Lagrangian submanifold of $M \times \Nbar,$ $s_i$ is a global covariant constant section of $(\cL_M \otimes \cL_N^*)|_{\Lambda_i},$ and, as is usual in homology, for any oriented submanifold $\Lambda\subset M \times \Nbar,$ we identify $-[\Lambda]$ with the submanifold $\Lambda$ with the opposite orientation.

There is an obvious composability condition for elements of these categories, but where they are composable, the orientations and sections compose nicely, using for the sections the formula (\ref{compeq}).

The advantage of the linear categories is that they are more directly associated with the currents dual to elements of our categories $\Gamma^*$ and $D^*.$  We do not belabor the point, since, as we shall see, the composition of currents apparently requires more structure on the submanifolds, corresponding perhaps to the usual choice of half-forms, or half-densities, in geometric quantization.  (See Remark \ref{halfdrem}.)  We will not address that issue in this paper.

\section{The  complex $\Gamma^*$ and its cohomology $D^*$ as duals to Weinstein's category and its integral avatar }\label{sec:complex}

Let $(M^{2m},\omega)$ be a compact, connected, integral symplectic manifold. Let
$\cL\to M$ be a Hermitian complex line bundle, and $\nabla$ a unitary connection on $\cL$
with curvature $\omega$. We denote the inner product on the fibres of $\cL$ by
$\langle\,\cdot\,,\,\cdot\,\rangle_\cL$. (As we noted before, the quadruple $(M,\omega,\cL,\nabla)$ is called a
\emph{prequantum system}.) Reversing the symplectic form gives a dual prequantum system
\[
  \Mbar = (M,-\omega,\cL^*,\nabla^+).
\]

\noindent where $\nabla^+$ is the connection on $\cL^*$ arising from the connection $\nabla$ and the hermitian structure.

Given such a system, consider the differential forms
\[
  \Omega^*(M,\cL) = \Omega^*(M)\otimes \Gamma(\cL).
\]
The connection $\nabla$ induces an operator we also denote~$\nabla$
\[
  \nabla\colon \Omega^*(M,\cL) \longrightarrow \Omega^{*+1}(M,\cL),
\]
with
\[
  \nabla^2 x \;=\; -i \omega\wedge x
\]
for all $x\in \Omega^*(M,\cL)$.

Likewise the
hermitian structure on $\cL$ gives a contraction
$$c: \Omega^j(M,\cL)  \times \Omega^k(M,\cL) \to \Omega^{j+k}(M)$$
\noindent for all $j,k.$  There is also a natural contraction 

$$\hc: \Omega^j(M,\cL^*)  \otimes \Omega^k(M,\cL) \to \Omega^{j+k}(M)$$
\noindent for all $j,k.$\footnote{Note $c$ is $\C-$antilinear in the first variable and $\C-$linear in the second, while $\hc$ is $\C-$linear in both variables.}

Consider now the subspace
\[
  \Gamma^*(M,\cL) \subset \Omega^*(M,\cL)
\]
given by
\[
  \Gamma^*(M,\cL) = \ker\bigl\{\omega\wedge\cdot\bigr\}.
\]
The operator $\nabla$ preserves $\Gamma^*(M,\cL)$ and makes it into a complex.

\begin{proposition}\label{prop:1.1}
$\Gamma^k(M,\cL) = 0$ unless $m \le k \le 2m$.
\end{proposition}

\begin{proof} 
On $\Omega^*(M,\cL)$ we have seen we have an action of the Lefschetz $\sla(2,\R)$
with
\begin{align*}
  L &= \omega\wedge, \\
  \Lambda &= \star_s L \star_s, \\
  H &= [\Lambda,L],
\end{align*}
where $H|_{\Omega^k(M,\cL)} = (m-k)\,\mathrm{1}$.

Thus $\Omega^*(M,\cL)$ decomposes into irreducible representations of $\sla(2,\R).$  But on each
such irreducible representation, $\ker L =\{0\}$ on subspaces of
$H$-grading greater than  $0$, that is, on differential forms below middle dimension.
\end{proof}

In middle dimension, the coisotropy and isotropy conditions are equivalent; so it is not surprising that
\begin{proposition}\labell{middimLlam}

In middle dimension we have

\[
{\rm ker}~ L|_{\Omega^m(M,\cL) } = {\rm ker}~ \Lambda |_{\Omega^m(M,\cL) }. \]
\end{proposition}
\begin{proof} Decompose $\Omega^*(M,\cL)$ into irreducible representations of $\sla(2,\R).$  On each irreducible, the kernel of $L$ consists of highest weight vectors.\footnote{Highest weight with respect to the differential form grading--these are actually lowest weight with respect to the grading by $H.$}  Thus, in middle dimension, the only irreducible where $L$ may act trivially is the trivial representation, where $\Lambda$ acts trivially as well.  \end{proof}

\begin{definition} \labell{deriv}  We write

$$D^*(M,\cL) = H^*(\Gamma^*(M,\cL), \nabla)$$

\noindent for the cohomology of the complex $(\Gamma^*(M,\cL), \nabla).$  \end{definition}

Note that

$$D^m(M,\cL) = {\rm ker}~\nabla \subset \Gamma^m(M,\cL).$$

\begin{theorem}\labell{Lagrangian}
The cohomology $D^*(M,\cL)$ is concentrated in middle dimension $m.$ \end{theorem}

\begin{proof}

Suppose we are given $\xi \in \Gamma^k(M,\cL)$ with $\nabla \xi = 0.$  We have, by item (10) of Lemma \ref{comms},

$$\{ \nabla,\nabla^\Lambda\} \xi = i (m-k) \xi.$$

\noindent But since $\nabla \xi = 0,$ 

$$\{ \nabla,\nabla^\Lambda\} \xi = \nabla \nabla^\Lambda \xi.$$

\noindent Hence, as long as $k\neq m,$ 

$$\xi =\nabla \left(\frac {-i}{(m-k)} \nabla^\Lambda \xi\right)$$

\noindent so that $\xi$ is exact.\footnote{Note $\frac {-i} {(m-k)} \nabla^\Lambda \xi \in {\rm ker~} L$ since $\xi \in {\rm ker}~ \nabla.$}  \end{proof}

\subsection{The complex $\Gamma^*$ and the cohomology $D^*$ as categories}\label{sec:ch2}

Let 

$$(M^{2m}, \omega_M, \cL_M, \nabla_M),  (N^{2n}, \omega_N, \cL_N, \nabla_N), (P^{2p}, \omega_P, \cL_P, \nabla_P)$$
\noindent denote three prequantum systems.  

Denote by
\begin{align*}
  i &\colon M\times N\times P \to M\times N,\\
  j &\colon M\times N\times P \to N\times P,\\
  k &\colon M\times N\times P \to M\times P,
\end{align*}
the projections, and let
\[
  k_* \colon \Omega^*(M \times N \times P,\cL_M \otimes \cL_P^*)\to \Omega^*(M \times P,\cL_M \otimes \cL_P^*)
\]
denote fibre integration.\footnote{Note that a symplectic manifold comes with a natural orientation.} 

Given sections
\[
  s\in \Gamma\bigl(i^*(\cL_M\otimes \cL_N^*)\bigr),\qquad
  s'\in\Gamma\bigl(j^*(\cL_N\otimes \cL_P^*)\bigr),
\]
we obtain a contraction
\[
  \hc_N (s \otimes s') \in \Gamma\bigl(k^*(\cL_M\otimes \cL_P^*)\bigr),
\]
and similarly for differential forms.

\begin{definition}\label{def:circ}
Given
\[
  \xi \in \Omega^*(M\times\Nbar,\, \cL_M\otimes \cL_N^*), \qquad
  \eta \in \Omega^*(N\times\Pbar,\, \cL_N\otimes \cL_P^*),
\]
define
\[
  \xi\circ\eta \in \Omega^*(M\times\Pbar,\, \cL_M\otimes \cL_P^*)
\]
by
\[
  \xi\circ\eta \;=\; k_* \hc_N\bigl(i^*\xi \wedge j^*\eta\bigr).
\]
\end{definition}

\begin{theorem}\label{leibniz}
Suppose we are given three prequantum systems arising
from symplectic manifolds $M, N, P,$ with prequantum line bundles $\cL_M, \cL_N, \cL_P.$
Let $\sigma\in\Gamma^*(M\times\Nbar,\, \cL_M\otimes \cL_N^*)$ and
$\zeta\in\Gamma^*(N\times\Pbar,\, \cL_N\otimes \cL_P^*)$. Then 
\[
  \nabla(\sigma\circ\zeta) \;=\; \nabla\sigma\circ\zeta \;+\; (-1)^{|\sigma|}\sigma\circ\nabla\zeta.
\]

\end{theorem}

Since $\nabla^2 = - i L,$ we have
 \begin{corollary}\labell{Lcomp}
Suppose we are given three prequantum systems arising
from symplectic manifolds $M, N, P,$ with prequantum line bundles $\cL_M, \cL_N, \cL_P.$
Let $\sigma\in\Gamma^*(M\times\Nbar,\, \cL_M\otimes \cL_N^*)$ and
$\zeta\in\Gamma^*(N\times\Pbar,\, \cL_N\otimes \cL_P^*)$. Then
\[
 L(\sigma\circ\zeta) \;=\; L\sigma\circ\zeta \;+ \sigma \circ L \zeta.
\]
\end{corollary}
Corollary \ref{Lcomp} implies
\begin{corollary} \labell{compos}
Suppose we are given three prequantum systems arising
from symplectic manifolds $M, N, P,$ with prequantum line bundles $\cL_M, \cL_N, \cL_P,$ and
elements
\begin{align*}
  \sigma &\in \Gamma^*(M\times\Nbar,\, \cL_M\otimes \cL_N^*),\\
  \tau   &\in \Gamma^*(N\times\Pbar,\, \cL_N\otimes \cL_P^*).
\end{align*}

Then the composition 

\[
  \sigma\circ\tau \in \Omega^*(M\times\Pbar,\, \cL_M\otimes \cL_P^*)
\]

lies in $ \Gamma^*(M\times\Pbar,\, \cL_M\otimes \cL_P^*).$\end{corollary}

Then Theorem \ref{leibniz} gives
\begin{corollary}\labell{compo}
The composition $\circ$ descends to the cohomology $D^*.$
\end{corollary}

\begin{proof}[Proof of Theorem~\ref{leibniz}.]
It suffices to prove this in the case where
\[
  \sigma = \alpha_M\wedge\beta_N , \qquad
  \zeta = \gamma_N\wedge\delta_P 
\]
\noindent and where\footnote{We again abuse notation by identifying these forms with their pullbacks to the appropriate product manifolds.} 
$$\alpha_M \in \Omega^*(M,\cL_M), \beta_N \in \Omega^*(N,\cL_N^*),  \gamma_N \in \Omega^*(N,\cL_N), \delta_P \in \Omega^*(P,\cL_P^*).$$ 

In this case
\[
  \sigma\circ\zeta = \alpha_M\wedge \delta_P
                   \int_N \hc_N( \beta_N\wedge \gamma_N)
\]

\noindent where we continue to write $\hc_N: \Omega^*(N,\cL_N^*)  \otimes \Omega^*(N,\cL_N) \to \Omega^{*}(N)$ for the natural contraction.

We then have
 
\begin{equation}
  (\nabla_M + \nabla_P^+)(\sigma\circ\zeta)
   = \left(\int_N \hc_N( \beta_N\wedge \gamma_N) \right)\left( \nabla_M \alpha_M   \wedge \delta_P  
    + (-1)^{|\alpha_M|  } \alpha_M \wedge   \nabla_P^+ \delta_P  \right)
 \end{equation}

On the other hand 
 
\begin{align*}
 ( (\nabla_M+\nabla_N^+) (\alpha_M\wedge\beta_N))\circ (\gamma_N \wedge \delta_P)
    +  (-1)^{|\alpha_M| + |\beta_N|}(\alpha_M\wedge\beta_N) \circ((\nabla_N+\nabla_P^+)(\gamma_N \wedge \delta_P))=& \\
     \nabla_M \alpha_M (\int_N\hc_N(\beta_N\wedge\gamma_N)) \wedge \delta_P +
    (-1)^{|\alpha_M|} \alpha_M (\int_N \hc_N(\nabla_N^+ \beta_N \wedge \gamma_N)) \wedge\delta_P +&\\
     (-1)^{|\alpha_M| +|\beta_N|} \alpha_M (\int_N \hc_N( \beta_N \wedge\nabla_N \gamma_N)) \wedge \delta_P +
 (-1)^{|\alpha_M| +|\beta_N| +|\gamma_N|} \alpha_M (\int_N \hc_N( \beta_N \wedge  \gamma_N))\wedge \nabla_P^+&\delta_P
\end{align*}
Since 
\begin{equation}
  0= \int_N d \hc_N(\beta_N\wedge \gamma_N )
    = \int_N  \hc_N( \nabla_N^+\beta_N  \wedge \gamma_N  )
     +  (-1)^{|\beta_N|}\int_N  \hc_N(\beta_N \wedge \nabla_N \gamma_N  ),
\end{equation} 
\noindent and noting that $|\beta_N| + |\gamma_N| = 2n$ is even, we see that the two expressions agree.
\end{proof}
We now check that the composition map $\circ$ is associative.
\begin{proposition}\labell{assoc}
Suppose we are given four prequantum systems corresponding to symplectic manifolds $M, N, P,Q, $ with prequantum line bundles $\cL_M, \cL_N, \cL_P,\cL_Q.$  Then for any
\(\sigma \in \Gamma^*(M\times\Nbar,\, \cL_M\otimes \cL_N^*)\),
\(\tau   \in \Gamma^*(N\times\Pbar,\, \cL_N\otimes \cL_P^*)\), and
\(\rho   \in \Gamma^*(P\times\Qbar,\, \cL_P\otimes \cL_Q^*)\);
we have
\begin{equation}\labell{assoceq}
  (\sigma\circ\tau)\circ\rho \;=\; \sigma\circ(\tau\circ\rho).
\end{equation}
 
\end{proposition}

\begin{proof}
It is sufficient to prove (\ref{assoceq}) when 

$$\sigma = \alpha_M \wedge \beta_N, \alpha_M \in \Omega^*(M ,\cL_M),  \beta_N \in \Omega^*(N, \cL_N^*)$$

$$\tau = \gamma_N \wedge \delta_P,  \gamma_N \in \Omega^*(N, \cL_N ), \delta_P \in \Omega^*(P,\cL_P^*)$$

$$\rho = \epsilon_P \wedge \zeta_Q,\epsilon_P  \in \Omega^*(P, \cL_P), \zeta_Q \in \Omega^*(Q,\cL_Q^*).$$

In that case, both the right hand and left hand sides of (\ref{assoceq}) are given by

$$(\sigma\circ\tau)\circ\rho =  \alpha_M \wedge \zeta_Q \int_N (\hc_N(  \beta_N\wedge  \gamma_N)) \int_P(\hc_P(\delta_P\wedge \epsilon_P)) = \sigma\circ(\tau\circ\rho),$$

\noindent where $\hc_P: \Omega^*(P,\cL_P^*)  \otimes \Omega^*(P,\cL_P) \to \Omega^{*}(P)$ is the natural contraction.

\end{proof}

We now have the building blocks for interpreting the complex $\Gamma^*$ and its cohomology $D^*$ as morphisms in categories.  One element is missing:  Unlike the Weinstein "category", where the diagonal Lagrangian in $M \times \Mbar$ is the identity morphism, our category does not have such a morphism.  To get around this, simply adjoin to the morphisms arising from $\Gamma^*$ an identity morphism $1_M$ for every symplectic manifold; in fact, one may as well adjoin a morphism $S_M$ for every symplectomorphism $S : M \to M$ preserving the hermitian line bundle $\cL$ and the connection 
$\nabla.$\footnote{Such symplectomorphisms are sometimes called {\em quantomorphisms}.  Note that we require the hermitian structure on the prequantum line bundle to be preserved, which is not necessarily standard in the literature.}  In the rest of the paper we suppress this issue to avoid further complicating the statements.

As in the Weinstein construction, our category $\Gamma^*$ has a form of involutivity:   For any two  prequantum systems arising
from symplectic manifolds $M, N $ with prequantum line bundles $\cL_M, \cL_N, $\ and given $\sigma\in\Gamma^*(M \times\Nbar,\, \cL_M\otimes \cL_N^*)$, we write
\[
  \sigma^T \in \Gamma^*(N\times\Mbar,\, \cL_N\otimes \cL_M^*),
\]
for the form arising from $\sigma$ via the map $M \times N \to N \times M$ exchanging $M$ and $N$ and the hermitian structures on $\cL_M$ and $\cL_N.$\footnote{Note that $\cdot^T$ is $\C-$antilinear.}  It is then clear that $(\sigma^T)^T = \sigma$ for any  $\sigma\in\Gamma^*(M \times\Nbar,\, \cL_M\otimes \cL_N^*).$  This involutivity is graded contravariant with respect to the operation $\circ:$  If $ \sigma \in \Gamma^*(M\times\Nbar,\, \cL_M\otimes \cL_N^*),
\tau   \in \Gamma^*(N\times\Pbar,\, \cL_N\otimes \cL_P^*)$ as above, then

$$(\sigma \circ \tau)^T = (-1)^{|\sigma||\tau|} \tau^T \circ \sigma^T.$$  

We summarize these results in the following.

\begin{theorem}\labell{gammacat}
There exists a $\Z_2$-graded involutive differential graded category $\Gamma^*$ whose objects are prequantum systems $(M,\omega,\cL,\nabla)$, and whose
morphisms $N\Rightarrow M$ between two such systems $(M,\omega_M,\cL_M,\nabla_M)$ and $(N,\omega_N,\cL_N,\nabla_N)$ are given by the complex
\[
  \Gamma^*(N,M) := \Gamma^*(M\times\Nbar,\, \cL_M\otimes \cL_N^*)
\]
with the composition law~$\circ.$ \end{theorem}

Note that if ${\rm dim~} M = 2m , $ ${\rm dim} ~N = 2n,$ and ${\rm dim}~ P = 2p, $ compositions of elements in $\Gamma^{m + n} (M,N)$ with elements
in $\Gamma^{n+p}(N,P)$ give elements of $\Gamma^{m + p}(M,P)$; so that the category
$\Gamma^*$ contains a subcategory with morphisms
\[
  \Gamma_{Lag}(N,M) := \Gamma^{m+n} (M^{2m} \times\Nbar^{2n},\, \cL_M\otimes \cL_N^*)
\]
corresponding to (duals of) Lagrangian submanifolds.   The involution of the differential graded category, which is graded contravariant, is contravariant in the linear sense if the middle dimensional forms are even, that is, if the dimension of the manifold is a multiple of four.

\begin{theorem}\labell{laggammacat}
There exists a linear category $\Gamma_{Lag}$ whose objects are prequantum systems $(M^{2m},\omega,\cL,\nabla)$, and whose
morphisms $N^{2n} \Rightarrow  M^{2m} $ between two such systems $(M^{2m},\omega_M,\cL_M,\nabla_M)$ and $(N^{2n},\omega_N,\cL_N,\nabla_N)$ are given by the vector space
$$\Gamma_{Lag} (N^{2n},M^{2m}): =\Gamma^{m+n}(M\times\Nbar,\, \cL_M\otimes \cL_N^*)
$$
with the composition law~$\circ.$  Restricted to manifolds whose dimension is a multiple of $4,$ this category is involutive.
\end{theorem}

Theorem \ref{leibniz} shows that this structure descends to the cohomology $D^*.$   In view of the vanishing theorem (Theorem \ref{Lagrangian}), we have
 
\begin{theorem}\labell{dcat}
There exists a linear category $D^*$ whose objects are prequantum systems $(M^{2m},\omega,\cL,\nabla)$, and whose
morphisms $N\Rightarrow M$ between two such systems $(M^{2m},\omega_M,\cL_M,\nabla_M)$ and $(N^{2n},\omega_N,\cL_N,\nabla_N)$ are given by the cohomology
\[
   D_{Lag}(N,M): = D^{m+n}(M\times\Nbar,\, \cL_M\otimes \cL_N^*)
\]
with the composition law~$\circ.$  Restricted to manifolds whose dimension is a multiple of $4,$ this category is involutive.

\end{theorem}

\begin{Remark}  Since $D^m(M,\cL) = \ker(\nabla) \subset \Gamma^m(M,\cL)$, we may view cohomology classes as forms; in other words, there is a map from the complex $D^*(M,\cL),$ consisting of the single group $D^m(M,\cL),$ with trivial differential, to the complex $\Gamma^*(M,\cL),$ which induces an isomorphism in cohomology.  Since this map is an inclusion, it respects the category structure.  In other words, the differential graded category $\Gamma^*$ is {\em formal} in the sense of being quasi-isomorphic to its cohomology category, and this quasi-isomorphism is given by a single map. \end{Remark} 
\begin{Remark}
The role of isotropics in quantization was recently investigated by Guillemin, Uribe,
and Wang~\cite{GUW}.  In our category $D^*$ they presumably correspond to exact forms.
\end{Remark}

In particular, let $\mathrm{pt}$ be the point, considered as a prequantum system. Consider the morphism spaces $ \mathrm{Mor}(\mathrm{pt},M),$ both in the complex $\Gamma^*(M,\cL)$ and its cohomology
 $D^m(M,\cL).$ 
 
In either case, given $\sigma\in \mathrm{Mor}(\mathrm{pt},M)$, we obtain an element
\[
  \sigma^T\circ\sigma \in \mathrm{Mor}(\mathrm{pt},\mathrm{pt}) = \C
\]
given by
\[
  \sigma^T\circ\sigma \;=\; \int_M c({\sigma} , \sigma).
\]

This pairing is nonzero only in middle dimension. In terms of dual currents, in the integral case, we will see below that this pairing is a dual of a type of Blattner-Kostant-Sternberg pairing on integral Lagrangian submanifolds of $M.$   Our results show that this pairing has cohomological content.  This fact might give an avenue for studying Lagrangian intersections.

\subsection{Promoting $\Gamma^*$ to a $\Z$-graded differential graded category}\labell{prom}  It is possible to regrade the category $\Gamma^*$ to obtain a $\Z$-graded differential graded category.  This also requires the twisting of the operation $\circ$ by a sign.  We summarize this construction briefly.

Let $(M^{2m}, \omega, \cL,\nabla)$ be a prequantum system. For any differential form $\sigma \in \Gamma^*(M,\cL),$ define a new grading $||\sigma||$ by

$$||\sigma|| = |\sigma| - m.$$

Note this is the grading given by the operator $-H,$ where $H$ is defined in (\ref{hdef}).  We denote the regraded complex by 

$$\Gamma^{-H}(M, \cL).$$

This regrading makes the composition $\circ$ have degree zero, which is not the case in the differential form grading.  Unfortunately, this change in grading destroys the Leibniz rule for the composition $\circ.$  To correct this, we define a new composition operation $\bullet.$  Suppose we are given three prequantum systems arising
from symplectic manifolds $M, N, P,$ with prequantum line bundles $\cL_M, \cL_N, \cL_P.$
Let $\sigma\in\Gamma^*(M\times\Nbar,\, \cL_M\otimes \cL_N^*)$ and
$\tau\in\Gamma^*(N\times\Pbar,\, \cL_N\otimes \cL_P^*)$. 
We define
$$\sigma \bullet \tau  = (-1)^{(m+n)|\tau|} \sigma \circ \tau.$$

It is then straightforward to check that the composition operator $\bullet$ satisfies the Leibniz rule with sign given by the regrading $||\cdot||.$

This regrading produces an anomaly in involutivity:  We have for $\sigma, \tau$ as above,

$$(\sigma \bullet \tau)^T = (-1)^{ (m+n)(n+p) }(-1)^{||\sigma||~||\tau||} \tau^T \bullet \sigma^T,$$

\noindent so involutivity holds when restricting to manifolds whose dimension is a multiple of four.

Then we obtain the following result:

\begin{reptheorem}{gammacat}[Second version]
There exists a $\Z$-graded differential graded category $\Gamma^{-H}$ whose objects are prequantum systems $(M,\omega,\cL,\nabla)$, and whose
morphisms $N\Rightarrow M$ between two such systems $(M,\omega_M,\cL_M,\nabla_M)$ and $(N,\omega_N,\cL_N,\nabla_N)$ are given by the complex
\[
  \Gamma^{-H}(N,M) := \Gamma^{-H}(M\times\Nbar,\, \cL_M\otimes \cL_N^*)
\]
with the composition law~$\bullet.$  Restricted to manifolds whose dimension is a multiple of $4,$ this category is involutive.

\end{reptheorem}

As in Remark \ref{meta} one may ask whether this involutivity anomaly can be corrected by some sort of metaplectic factor.

Since the Lagrangian category $\Gamma_{Lag}$ and the cohomology category $D^*$ are supported in a single dimension, it is not useful to regrade them.  The regrading by $||\cdot||$ simply shifts the Lagrangians to degree zero.

\subsection{Dual currents of $\Gamma^*(M,\cL)$ and $D^*(M,\cL)$ and the relation to the Weinstein "category"}\labell{reltow}

Let $(M,\omega,\cL,\nabla)$ be a prequantum system.  Let $j: C\to M$ be a compact, oriented submanifold of $M$, and let
$s\in\Gamma(\cL|_C).$
The $\cL$-valued de Rham current $\xi_{(C,s)}$ is defined by
\begin{equation}\labell{curdef}
  \xi_{(C,s)}(\eta) \;=\; \int_C c(s, j^*\eta)
\end{equation}

\noindent for any $\eta \in \Omega^*(M,\cL).$  Here we denote by $c$ the contraction on the fibres of $\cL$ induced by the hermitian structure.

The operator $\nabla$ induces a dual operator we also denote by $\nabla$ on the space of $\cL$-valued currents.  We now characterize the $\nabla-$closed currents:

\begin{proposition}\labell{prop:1.2}
Suppose $C \subset M$ is connected.  A nonzero current $\xi_{(C,s)}$ is closed (with respect to the dual of $\nabla$) if and only if
\begin{itemize}
  \item $C$ is isotropic,
  \item $\partial C = \emptyset$,
  \item $\nabla s = 0$.
\end{itemize}

\end{proposition}

\begin{proof}
Let $\eta\in \Omega^*(M,\cL)$. Then
\begin{align*}
  \xi_{(C,s)}(\nabla \eta)
   &= \int_C c(s, \nabla\eta) \\
   &= \int_C dc (s,\eta) -  \int_{C}c(\nabla s, \eta) 
       \\
   &= \int_{\partial C}c(s, \eta) -  \int_Cc( \nabla s, \eta).
\end{align*}
This vanishes for all $\eta$ if and only if $\partial C=\emptyset$ and $\nabla s = 0$.

Since $s$ is a global covariant constant section of $\cL|_C,$ $\cL|_C$ must be flat, so that $C$ must be isotropic.

\end{proof}

Note that the condition that $\nabla^2 \xi_{(C,s)} = 0$ is that $\xi_{(C,s)}$ vanishes on any form in the image of $L.$  This is equivalent to $C$ being isotropic on the support of $s.$ There is a similar condition on $\Lambda,$ giving rise to coisotropic currents (see \cite{TY}, Lemma 4.1).  Thus

\begin{proposition} Suppose $C  \subset M$ is a connected submanifold of $M$ and $s$ is a section of $\cL|_C.$  A nonzero current $\xi_{(C,s)}$ is in the kernel of the dual of $L$ if and only if $C$ is isotropic on the support of $s,$ and in the kernel of the dual of $\Lambda$ if and only if $C$ is coisotropic on the support of $s.$  It is in the kernel of the dual of $\nabla$ if and only if $\partial C = \emptyset,$ $C$ is isotropic and integral, and the section $s$ is covariant constant.
\end{proposition}
 
 \begin{Remark}[Coisotropics]  One can also consider the complex of differential forms in the kernel of $\Lambda,$ and its cohomology with respect to $\nabla^\Lambda.$  The symplectic star operator $\star_s$ gives an isomorphism of this complex with the complex $\Gamma.$  One can show that the symplectic star operation also factors across the composition operation $\circ.$  So we obtain a coisotropic complex.  If we are given a current $\xi_{(C,s)}$ as above, with $C$ now coisotropic, we calculate as in the proof of Proposition \ref{prop:1.2}, for any $\eta \in \Omega^{{\rm dim~} C +1}(M,\cL),$
 
\begin{align*} 
  \xi_{(C,s)}(\nabla^\Lambda \eta)
   &= \int_C c(s,(\nabla \Lambda-\Lambda \nabla) \eta ) \\
   &= - \int_Cc(s, \Lambda \nabla \eta ) +
 \int_{\partial C}c(s,\Lambda \eta ) -  \int_Cc(\nabla s, \Lambda\eta ) \end{align*}
The first term on the right side of this equation vanishes due to coisotropy of $C;$ the second term vanishes if $C$ has coisotropic boundary.  If $C$ has dimension $m$, then since $\Lambda$ is surjective onto forms of dimension $m-1,$ the third term vanishes if and only if $s$ is covariant constant, so that $C$ must be an integral Lagrangian submanifold. But it is not so easy to see what the third term gives where ${\rm dim}~ C > m.$  Note that the symplectic star duals of the integral isotropic currents $\xi_{(C,s)}$ are also currents in this complex, and its cohomology, but they are not associated to submanifolds as in (\ref{curdef}).
 \end{Remark}
 
 The complex $\Gamma^*(M,\cL)$ is therefore dual to oriented isotropic submanifolds of $M$, equipped with sections of the prequantum line bundle, and considered as de Rham currents.  The cohomology $D^m(M,\cL)$ is dual to oriented integral Lagrangian submanifolds of $M,$ equipped with covariant constant sections of the prequantum line bundle.\footnote{Integral isotropic submanifolds also appear in this setting, but they  presumably correspond to exact currents.}
These submanifolds play an important role in geometric quantization, where Lagrangian
submanifolds arise as fibres of a Lagrangian foliation
\[
  \pi\colon M^{2m}\to B^m
\]
of $M$ over a base $B^m$; the integral fibres, known in this case as Bohr--Sommerfeld
fibres, form a basis for geometric quantization.
The complex $(\Gamma^*(M,\mathcal{L}),\nabla)$ and its cohomology
furnish a natural setting where these submanifolds arise (as duals).

We briefly consider what the composition morphism $\circ$ of our dual categories would mean for currents of the form $\xi_{(C,s)},$ where $C$ is a Lagrangian (or even isotropic) submanifold of a product $M\times \Nbar,$ and $s$ is a section of the bundle $(\cL_M \otimes \cL_N^*)|_C$ (any section, if we are dealing with $\Gamma^*(M\times \Nbar,\cL_M \otimes \cL_N^*)$; a covariant constant section, in the case of $D^*(M\times \Nbar,\cL_M \otimes \cL_N^*).$)

Suppose we are given three prequantum systems arising
from symplectic manifolds $M, N, P,$ with prequantum line bundles $\cL_M, \cL_N, \cL_P,$ and
the currents $\xi_{(C,s)}, \xi_{(C',s')}$ supported on isotropic submanifolds $C \subset M \times \Nbar,$
$C' \subset N \times \Pbar$ and sections $s \in \Gamma ((\cL_M \otimes \cL_N^*)|_C), s' \in \Gamma ((\cL_N \otimes \cL_P^*)|_{C'}).$  These sections are general for the category $\Gamma^*$ and covariant constant for the case of the category $D^*.$

These currents give us currents on $M \times N \times P$ supported on $C \times P$ and $M \times C'.$ Taking the product of these two currents gives us a current supported on points $C \star C' \subset M \times N \times P$ of the form $(m,n,p)$ where $(m,n) \in C$ and $(n,p) \in C'$.  Projecting $M \times N \times P$ to $M \times P$ gives us, in the case where $C$ and $C'$ are composable, a current supported on the Weinstein composition 
$$C \circ C' = \{(m,p) \in M \times P:  (m,n) \in C {\rm ~and~} (n,p) \in C' {\rm~for~some~} n \in N\}.$$  

Thus the dual composition operation is closely related to Weinstein's composition, where it is defined.  But in the case where $C$ and $C'$ are not composable, the composition in the dual category is still well defined; if we imagine the currents  $\xi_{(C,s)}, \xi_{(C',s')}$ as being smoothed out to differential forms, the composite differential form is still supported near the set $C \circ C',$ but the value of the differential form is obtained by integration over the set of points of $N$ lying above a point of $C \circ C'$ in $C \star C'.$  Thus the fact that we have a linear category, rather than a category of manifolds, allows us to sum, or integrate, elements of our category, to deal with more complicated intersections.

\begin{Remark}\labell{halfdrem} To actually compose currents, it is necessary to smooth them to differential forms.  Each such differential form gives rise to a current via the map 

$$\xi \to \xi^T \circ (\cdot)$$

\noindent where we consider elements of $D^m(M,\cL)$ as morphisms from pt to $M.$  Lifting a current to an element of $D^*$ involves choices, so that the composition of currents may involve "decorating" the underlying isotropic with further data.  The appearance of half densities, or half forms, in the BKS pairing may possibly be interpreted as a form of that type of choice.\end{Remark}

Remark \ref{halfdrem} raises the following question about the Weinstein "category":

\begin{Question}  Since the categories $\Gamma^*$ and $D^*$ are true categories, is it possible to redefine composition in some (possibly decorated) version of the linear Weinstein "category" by deformations so it is defined even for noncomposable morphisms, giving a true category?\end{Question}

\begin{Remark} {Other constructions giving categories associated to Weinstein's construction are \cite{K,WW}.}\end{Remark}

\section{Application to geometric quantization}\labell{gq}
Geometric quantization would assign to a prequantum system $(M^{2m},\omega,\mathcal{L},\nabla)$
a vector space $Q(M),$ ideally equipped with a hermitian structure,  functorial under canonical
transformations, or perhaps some subset of these. The expectation that, when $M$ is
compact, $Q(M)$ should be finite dimensional, arising from many examples in physics and
representation theory, excludes the obvious vector space $\Gamma(\cL)$. Instead the
prequantum data must be supplemented by a polarization that cuts down $\Gamma(\cL)$ to some
subspace.  

One example of a polarization is a \emph{complex polarization}, where $(M,\omega)$ is Kahler, $\cL$ a
holomorphic line bundle, and $\nabla$ the Chern connection. In that case the quantization in this
holomorphic polarization is taken as the space of holomorphic sections of $\cL$
\[
  Q(M) = H^0(M,\mathcal{L}).
\]

Another type of polarization is a \emph{real polarization}, a foliation of $M$ by
Lagrangian subvarieties given by a map $\pi\colon M\to B$. In this case the quantization $Q(M)$ is
given by distributions associated to covariant constant sections of $\mathcal{L}$ localized on those leaves $\pi^{-1}(b)$ of the foliation
where $(\mathcal{L},\nabla)|_{\pi^{-1}(b)}$ is trivial as a line bundle with connection.
These are the integral isotropic leaves of the foliation, called the Bohr--Sommerfeld leaves.

Where the foliation is fibering, Sniatycki's Theorem \cite{S} shows that this quantization arises
from middle-dimensional sheaf cohomology associated with local covariant constant sections
of $\mathcal{L},$ and that the sheaf cohomology vanishes at all other degrees.  In view of the vanishing theorem of Theorem \ref{Lagrangian}, this shows an intriguing relation to elements of the (dual of) our dual category $D^*.$

To begin with there would seem to be no relation between such a count of integral
Lagrangian leaves of a foliation and holomorphic sections of a line bundle. But in many
compact examples, it turns out that at least the dimensions of the spaces are equal.   Some examples are toric varieties \cite{fulton}, coadjoint
orbits~\cite{GSGC}, moduli spaces of vector bundles~\cite{JW}, and, in the noncompact
setting, Fourier Series and the Peter--Weyl Theorem~\cite{CW}.\footnote{In the noncompact case invariance of polarization is given by a natural isomorphism of the quantizations in the two different polarizations.}  In general there is an expectation that quantization should depend only
on the prequantum system and not on the polarization. But no such theorem exists in general.

\subsection{The Kahler case} We now show how our cohomology $D^*(M ,\mathcal{L})$, which we saw was associated with
integral Lagrangian submanifolds, is related, in the Kahler case, to geometric
quantization in the Kahler polarization. This provides a bridge between the world of Lagrangian submanifolds and the world of holomorphic quantization.

Let $(M^{2m}, \omega, \cL, \nabla)$ be a prequantum system where $M^{2m}$ is a compact connected symplectic manifold of dimension $2m.$  Recall that the complex $(\Gamma^*(M,\cL),\nabla)$ vanishes in degrees less than $m$ (or greater
than $2m$). So
\[
  D^m(M,\cL)= H^m\bigl(\Gamma^*(M,\mathcal{L}),\nabla\bigr)
  \;=\; \ker \nabla\big|_{\Omega^m(M,\mathcal{L})}.
\]

Recall the Lefschetz $\sla(2,\R)$ representation given by generators
\begin{align*}
  L &= \omega\wedge,\\
  \Lambda &= \star_s L \star_s\\
  H &= [ \Lambda,L] 
\end{align*}
 
where $H$ is the grading operator: $H|_{\Omega^k(M,\cL)} = m - k$.

In our twisted case, we also had operators
\[
  \nabla,\qquad \nabla^\Lambda =    [\nabla,\Lambda]
\]
with
\[
  \{\nabla,\nabla\} = - 2iL,\qquad \{\nabla^\Lambda,\nabla^\Lambda\} = 2i\Lambda.
\]
So as we saw the standard Lefschetz $\sla(2,\R)$ action on differential forms on a symplectic manifold is promoted to an action of the superalgebra ${\mathfrak {osp}} (1|2)$ on the twisted differential forms in the prequantum case.

In middle dimension we have, for any $x \in \Omega^m(M,\cL),$\footnote{The first implication is because $\nabla^2 =- i L;$  the second, by Proposition \ref{middimLlam}; the third, since $\nabla^\Lambda= [\nabla,\Lambda].$}
\[
  \nabla x = 0 \;\Rightarrow\; L x = 0 \;\Rightarrow\; \Lambda x = 0 \;\Rightarrow\;
  \nabla^\Lambda  x = 0.
\]
 
Consider now the case where $(M,\omega)$ is Kahler and $(\mathcal{L},\nabla)$ is a
holomorphic line bundle equipped with Chern connection $\nabla$.

Then we have the Dolbeault decomposition
\[
  \nabla = \nabla' + \nabla'',
\]
and its adjoint
\[
  \nabla^* = \nabla'^{\,*} + \nabla''^{\,*}
\]
where we have denoted by $\cdot^*$ the adjoint in the Kahler metric.

\begin{definition}\label{def:Qdef}

We define
\[
  Q(M,\mathcal{L}) \;=\; \ker\nabla \cap \ker\nabla^*.
\]
Since $\nabla^2 = - iL$ and $(\nabla^*)^2 =  i\Lambda$, $Q(M,\cL)$ consists of middle dimensional forms.\footnote{Of course this also follows from Theorem \ref{Lagrangian}.}

We write $Q(M) = Q(M,\mathcal{L})$ where there is no chance of confusion.
\end{definition}

\begin{Remark} It may be possible to view the two conditions giving $Q(M)$ in Definition \ref{def:Qdef} as an equation of motion and a gauge condition:  The condition $\nabla x = 0$ for $x\in \Gamma^m(M,\cL)$ is that
$x \in D^m(M,\cL)$, and is cohomological, while the condition $\nabla^* x = 0$ is an analog of covariant gauge, which depends on a choice of metric.  In that spirit, it may be possible to view a real polarization as giving a kind of axial gauge condition, perhaps restricting to forms which are basic with respect to the map $\pi: M \to B$ giving the real polarization.  \end{Remark}

To study $Q(M)$, we study the Laplacians
\begin{align*}
  \Delta &= \nabla^*\nabla + \nabla\nabla^*,\\
  \Delta_{\bar\partial} &= (\nabla'')^*\nabla'' + \nabla''(\nabla'')^*,\\
  \Delta_\partial &= (\nabla')^*\nabla' + \nabla'(\nabla')^*.
\end{align*}
Then $Q(M) = \ker\Delta \subset D^m(M,\mathcal{L})$.

\begin{proposition}\label{prop:laplacians}
The Laplacians $\Delta$, $\Delta_{\bar\partial}$, $\Delta_\partial$ satisfy
\begin{enumerate}
  \item[(i)] $ \Delta_{\bar\partial} + \Delta_\partial = \Delta $,
  \item[(ii)] $\Delta_{\bar\partial} - \Delta_\partial = - H$.
\end{enumerate}
\end{proposition}

\begin{proof} We have
\begin{align*}
  \Delta &= (\nabla'+\nabla'')^*(\nabla'+\nabla'') + (\nabla'+\nabla'')(\nabla'+\nabla'')^*\\
           &= \Delta_\partial + \Delta_{\bar\partial}
             + \nabla''^{\,*}\nabla' + \nabla'\nabla''^{\,*}\\
           &\qquad + \nabla'^{\,*}\nabla'' + \nabla''\nabla'^{\,*}\\
           &= \Delta_\partial + \Delta_{\bar\partial}
             + \{\nabla'^{\,*},\nabla''\} + \{\nabla''^{\,*},\nabla'\}.
\end{align*}
But the Kahler identities (See e.g. \cite{demailly} p.~329, 1.1~(c,d)) are\footnote{These follow from the standard (untwisted) Kahler identities by a method similar to that used in the proof of Lemma \ref{nablalambda}. Restrict to a ball where the bundle $\cL$ is trivialized.  To prove the first identity in (\ref{ki}), put
$E(\nabla) = [\Lambda,\nabla''] + i\,\nabla'^{\,*},$ so that the assertion is
$E(\nabla)=0,$ and $E(d)=0$ is the untwisted case. In a local unitary frame
$\nabla = d+A$ with $A$ purely imaginary, and
$E(\nabla)-E(d) = [\Lambda,A^{0,1}\wedge] + i\,(A^{1,0}\wedge)^*$
is $C^\infty$-linear; hence $E(\nabla)$ at a point depends only on the value of $A$ there.
Given $p\in M,$ choose $\phi$ real with $-i\,d\phi(p) = A(p)$ and set
$\nabla_\phi = e^{i\phi}\,d\,e^{-i\phi} = d - i\,d\phi.$ As $\Lambda$ is
$C^\infty$-linear and $e^{i\phi}$ is unitary, conjugation passes through both
$\Lambda$ and the adjoint, so $E(\nabla_\phi) = e^{i\phi}E(d)e^{-i\phi} = 0.$
Since $\nabla$ and $\nabla_\phi$ have the same connection form at $p,$
$E(\nabla)(p) = 0.$ The second identity follows by a similar argument.}
\begin{equation}\labell{ki}
  \nabla'^{\,*} = i[\Lambda,\nabla''],\qquad \nabla''^{\,*} = -i[\Lambda,\nabla'].
\end{equation}
So
\begin{equation}\labell{keq}\begin{split}
  \{\nabla'^{\,*},\nabla''\} + \{\nabla''^{\,*},\nabla'\}
   &= i\,\{[\Lambda,\nabla''],\nabla''\} - i\,\{[\Lambda,\nabla'],\nabla'\}\\
   &= i\bigl(\Lambda \nabla''\nabla'' - \nabla''\Lambda \nabla'' + \nabla''\Lambda \nabla'' - \nabla''\nabla''\Lambda\bigr)\\
   &\quad - i\bigl(\Lambda \nabla'\nabla' - \nabla'\Lambda \nabla' + \nabla'\Lambda \nabla' - \nabla'\nabla'\Lambda\bigr).
\end{split}\end{equation}
Since $\nabla^2 = -i\omega \wedge $ and $\omega$ is of type $(1,1),$  the operators $(\nabla')^2$ and $(\nabla'')^2,$ which are wedge products with the vanishing $(2,0)$ and $(0,2)$ components of the curvature, vanish, so the right hand side of (\ref{keq}) vanishes.

(ii) is the Bochner-Kodaira-Nakano identity; see e.g. \cite{demailly}, Chapter VII, p. 330, Thm~1.2, in the torsion-free case.  \end{proof}
\begin{Remark} Alternatively, (ii) follows from the Kahler identities (\ref{ki}):
these give $\Delta_\partial = i\{\nabla',[\Lambda,\nabla'']\}$ and
$\Delta_{\bar\partial} = -i\{\nabla'',[\Lambda,\nabla']\},$ so
$\Delta_{\bar\partial} - \Delta_\partial 
 = -i\bigl(\{\nabla'', [\Lambda,\nabla']\}+ \{\nabla',[\Lambda,\nabla'']\}\bigr)
 = -i\,[\Lambda,\{\nabla',\nabla''\}].$
Since $(\nabla')^2=0$ and $(\nabla'')^2=0,$ we have $\{\nabla',\nabla''\} = \nabla^2 = -iL,$
so that $\Delta_{\bar\partial}- \Delta_\partial=-  i[\Lambda,-iL] = - [\Lambda,L] = - H.$ 
\end{Remark}

Hence
\[
  \Delta = 2\,\Delta_{\bar\partial} +  H,
\]
and therefore, since for middle dimensional forms $H = 0,$  we obtain:

\begin{corollary}\label{cor:Q-kerdbar}
The quantization $Q(M)$ is given by
\[
  Q(M) \;=\; \ker \Delta_{\bar\partial}|_{\Omega^m(M,\cL)}.
\]
\end{corollary}

Note that $Q(M)$ is a hermitian vector space with inner product given by

\[
  \langle \xi,\eta\rangle \;=\; \int_M c(\xi ,*\eta) \in \C
\]
for $\xi,\eta\in \ker\nabla\cap \ker \nabla^*,$ where $*$ is the (Kahler) Hodge star operator.

We now have the following result, which gives a bridge between the dual category, which
lives in the Lagrangian world, and holomorphic quantization.

\begin{theorem}\label{kahlergq}
Let $(M,\omega)$ be a compact Kahler manifold, and let $(\mathcal{L},\nabla)$ be a
hermitian holomorphic line bundle with Chern connection of curvature $\omega.$ Let $k$ be sufficiently large
and write $\nabla_k$ for the connection on $\cL^k$ induced by $\nabla.$  Define
\[
  Q(M,\cL^k ) \;=\; \ker(\nabla_k)\cap \ker(\nabla_k^*) \;\subset\; \Gamma^*(M,\mathcal{L}^k).
\]
\noindent where we continue to write $\nabla_k$ for the operator induced by the connection $\nabla_k$ on $\Gamma^*(M,\mathcal{L}^k)$ and $\nabla_k^*$ for its adjoint.

Then
\[
  Q(M,\cL^k) \;\simeq\; H^0(M,\, \mathcal{L}^k \otimes K_M  ),
\]
where $K_M$ is the canonical bundle of $M$.
\end{theorem}

\begin{proof}
We have, by Corollary \ref{cor:Q-kerdbar} applied to the line bundle $\cL^k,$
\[
  Q(M,\cL^k) = \ker\bigl(\Delta_{\bar\partial_k}|_{\Omega^m(M,\cL^k)}\bigr)
       \subset \Omega^m(M,\mathcal{L}^k)
       = \bigoplus_{p+q=m}\Omega^{p,q}(M,\mathcal{L}^k)
\]

\noindent where 
$$ \Delta_{\bar\partial_k} = (\nabla_k'')^*\nabla_k'' + \nabla_k''(\nabla''_k)^*.$$

By the Hodge Theorem for $\Delta_{\bar\partial_k}$\footnote{See e.g. Wells \cite{Wells}, p. 151.}

\[
  \ker\bigl(\Delta_{\bar\partial_k}|_{\Omega^m(M,\cL^k)}\bigr) = \bigoplus_{p+q=m} H^q(M, {\bf\Omega}^p_M \otimes\mathcal{L}^k).
\]

\noindent where ${\bf\Omega}^*_M$ is the sheaf of holomorphic differential forms on $M.$

But for $k$ sufficiently large, $$H^q(M,\,{\bf\Omega}^p_M \otimes\mathcal{L}^k) =0 {\rm ~for~} q\neq0$$ for all $p,$ by Serre's vanishing theorem.

Hence, for $k$ sufficiently large,
\[
  Q(M,\cL^k)= H^0(M,\,{\bf\Omega}^m_M \otimes\mathcal{L}^k) = H^0(M,\,\mathcal{L}^k \otimes K_M).
\]

\end{proof}

\subsection{Beyond the Kahler case?} We have seen that our definition of $Q(M)$ using both $\nabla$ and $\nabla^*$ in the Kahler case gives a relation between the Lagrangian and holomorphic worlds.  We may wonder if something similar happens for general symplectic manifolds, which can always be equipped with a (not necessarily integrable) almost complex structure compatible with the symplectic form.
\begin{question}\labell{acs}
Let $(M,\omega,\cL,\nabla)$ be a prequantum system with $M$ compact and connected.
Let $J$ be an almost complex structure, compatible with $\omega$, and let $\bar\partial_J$
be the corresponding Dirac operator on the complex line bundle $\cL  \otimes K_M.$ 
Is it true that (replacing $\cL$ with a sufficiently high power $\cL^k$ if needed)
\[
  {\rm dim}( \ker \nabla \cap \ker \nabla^* )\;=\; \mathrm{ind}(\bar\partial_J),
\]
where $\nabla^*$ is the adjoint in the metric arising from $\omega$ and $J$?
\end{question}

A more speculative question is:

\begin{question}
Is there some way of defining $Q(M)$ without a choice of almost complex structure, which
will coincide with our definition in the Kahler case?
\end{question}

One path to answering this question would be further study of the twisted analog of the symplectic Hodge theory of
\cite{B,G,M, TY} in the
prequantum case.
The van Hove theorem indicates that such a quantization must involve some choices or have
some defects.  
\subsection{Real polarizations}We may obtain some geometric insight into these questions by considering what happens
when $M$ is equipped with a real polarization, that is, a Lagrangian foliation whose leaves are the fibres of a map $\pi\colon M\to B$ for
some $B$. In this case the integral leaves are given by $\pi^{-1}(b)$, where $b\in B_{BS}\subset B$, and $B_{BS}$ is a finite subset of $B.$  These leaves may be
expected to give linearly independent elements of a putative quantization.  To be concrete,
in the case where $M$ is also Kahler, and $(\cL,\nabla)$ holomorphic, we may expect the
$L_2$ projection of the de Rham currents corresponding to these leaves $\pi^{-1}(b)$, $b\in B_{BS}$,
with their covariant constant sections of $\mathcal{L}|_{\pi^{-1}(b)}$, considered as differential forms
lying in some negative Sobolev space, to the harmonic
forms, to be a basis for $Q(M).$  The fact that, at least for a sufficiently positive
line bundle, we obtain the holomorphic sections of that bundle, tensored with the canonical bundle, is an
indication that only integral Lagrangian leaves of a real polarization contribute to the
quantization, rather than the integral isotropic leaves that show up in such foliations more generally.  

\subsection{The case of toric varieties} A special case which may serve as a paradigm for what we may hope for in a general theorem is that of toric varieties, where we may compare quantization in a real polarization with holomorphic quantization, and see how the fact that $Q(M)$ is given by $H^0(M, \cL^k \otimes K_M),$ with the appearance of the canonical bundle, corresponds to a restriction of the real quantization to Lagrangian (not merely isotropic) Bohr Sommerfeld fibres.  

\def\bL{\mathbb L}
A toric variety\footnote{We assume integrality.} $(M^{2m},\omega,\cL,\nabla)$ is a Kahler manifold equipped with a holomorphic bundle $\cL$ with Chern connection $\nabla$ of curvature $\omega$ and with an effective Hamiltonian action of a compact torus $T^m$ of half the dimension of $M,$ which lifts to $\cL.$  The image $\cP = \mu(M)$ of the moment map $\mu: M\to {\mathfrak t}^*$ is a convex polytope, whose vertices lie in the integral lattice $\bL \subset {\mathfrak t}^*$ given by the weights of the group $T^m.$  The moment map foliates $M$ by isotropic submanifolds $\mu^{-1}(t),$ $t \in {\mathfrak t}^*,$ given by the orbits of the torus $T^m.$ This is a real polarization of $M.$ These isotropic tori are Lagrangian when $t \in {\rm Int}(\cP).$  

The quantization of $M$ in this real polarization is given by the distributional covariant constant sections of $\cL$ supported on the integral isotropic leaves of the foliation, given in this case by those fibres $\mu^{-1}(t)$ of the map $\mu$ where $t \in \bL.$  The subspace of the quantization corresponding to the Lagrangian leaves only has as a basis the covariant constant sections of $\cL$ supported on those Bohr-Sommerfeld leaves given by $\mu^{-1}(t)$ where $t \in  \bL \cap {\rm Int}(\cP).$

The relation of this real polarization picture to the Kahler quantization of $M$ is given by the fact that (see \cite{fulton})

$${\rm dim}~H^0(M,\cL) = \#(\cP \cap \bL)$$

\noindent while (see e.g. \cite{aw})

$${\rm dim}~H^0(M,\cL\otimes K_M ) = \#( {\rm Int}(\cP) \cap \bL).$$

Our Theorem \ref{kahlergq}, in the general case of any prequantum system associated to a Kahler manifold, gives a relation between the quantization $Q(M),$ which lives in the world of the dual category of Lagrangian submanifolds, and the holomorphic quantization $H^0(M,\cL^k \otimes K_M),$ mirroring the direct computation in the toric case.  The fact that Lagrangian (not general isotropic) quantization gives the holomorphic sections of the polarizing bundle, {\it tensored with $K_M,$} is therefore what we would expect from our study of the toric case.  A similar exclusion of isotropic but non-Lagrangian leaves in the context of quantization in a real polarization was found by
Hamilton \cite{ham}, who studied the sheaf cohomology $H^*(M,{\mathcal S})$ for locally toric prequantum systems $(M,\omega,\cL,\nabla)$ equipped with a real polarization, where ${\mathcal S}$ is the sheaf of fibrewise covariant constant sections of $\cL,$ as inspired by the results of Sniatycki for abelian varieties. 

To pose a concrete question, recall that Lagrangian Bohr-Sommerfeld leaves of a real polarization give currents dual to elements of our category.  Considering these currents as differential forms lying in some Sobolev space of distributions, we may, as we noted above, project them onto the Harmonic forms giving $Q(M),$ since those lie in any Sobolev space.   So we may consider these "Bohr-Sommerfeld currents" as elements of $Q(M).$  The many examples in which real and holomorphic polarizations yield isomorphic quantizations motivate the following question:\footnote{In many of these examples, the real quantization counts also the isotropic, non Lagrangian leaves of the polarization, and the Kahler quantization does not have the factor of $ \otimes K_M.$  We expect that the question is more natural if we exclude the non-Lagrangian Bohr-Sommerfeld leaves, and include the tensor product with the canonical bundle.}

\begin{question}
Let $(M,\omega,\cL,\nabla)$ be a prequantum system, where $M$ is a Kahler manifold, $\cL$ a holomorphic line bundle, and $\nabla$ the Chern connection.  Suppose $M$ also has a real polarization.  Do the currents corresponding to the Lagrangian Bohr--Sommerfeld leaves of the real polarization, projected to the harmonic forms, give a basis for $Q(M)$?
\end{question}

This projection of a distributional section of $\cL$ supported on some Bohr Sommerfeld Lagrangian to the harmonics should smear out to a $\Delta-$harmonic element of $\Omega^m(M,\cL),$ or equivalently, to a holomorphic
section of $\cL\otimes K_M,$ whose maximum modulus we would expect \cite{Guilleminpc} to lie
along the Lagrangian.


\begin{thebibliography}{99}

\bibitem[AW]{aw}Agapito, Jose; Weitsman, Jonathan. The weighted Euler-Maclaurin formula for a simple integral polytope.  
Asian J. Math. 9, No. 2, 199-211 (2005)
 
\bibitem[B]{B} J.-L.~Brylinski, A differential complex for Poisson manifolds.   {\em J. Diff. Geom.} 28, 93-114 (1988)
\bibitem[CW]{CW} Crooks, P., Weitsman, J., The double Gelfand-Cetlin system, invariance of polarization, and the Peter-Weyl theorem. Journal of Geometry and Physics, 194 (2023)
\bibitem[D]{demailly} J.-P. Demailly, {\it Complex analytic and differential geometry.}  Online text, retrieved from  \url{https://people.math.harvard.edu/~demarco/Math274/Demailly\_ComplexAnalyticDiffGeom.pdf}

\bibitem[FSS]{fss} L. Frappat, P. Sorba, A. Sciarrino, Dictionary on Lie Superalgebras.  arXiv:hep-th/9607161.
\bibitem[F]{fulton} W. Fulton, Introduction to Toric Varieties.  Princeton University Press.
\bibitem[GH]{gh} P. Griffiths, J. Harris.  {\it Principles of Algebraic Geometry.}  Wiley, 1978
\bibitem[G1]{G} V.~Guillemin, Symplectic hodge theory and the $d\delta$ lemma.  Preprint.
\bibitem[G2]{Guilleminpc} V.~Guillemin, personal communication.
\bibitem[GP]{GP} V.~Guillemin and A.~Pollack, \emph{\it Differential Topology}, Prentice-Hall, 1974.
\bibitem[GS1]{GSGC} Guillemin, Victor; Sternberg, Shlomo.  The Gel'fand-Cetlin system and quantization of the complex flag manifolds, Journal of Functional Analysis, 52 (1): 106-128 (1983)
 \bibitem[GS2]{GS} V.~Guillemin and S.~Sternberg, \emph{\it Semiclassical Analysis.}  International Press, 2013
\bibitem[GUW]{GUW} V.~Guillemin, A.~Uribe, and Z.~Wang, Integral representations of isotropic semiclassical functions and applications.  {\em J. Spectr. Theory} 12 (2022), no.1, 227-258.
\bibitem[H]{ham} M. Hamilton, Locally toric manifolds and singular Bohr-Sommerfeld leaves.  arXiv:0709.4058;
Mem. Amer. Math. Soc. 207 (2010), no. 971
\bibitem[JW]{JW} L.~Jeffrey and J.~Weitsman.  Bohr-Sommerfeld orbits in the moduli space of flat connections and the Verlinde dimension formula. Comm. Math. Phys. 150(3): 593-630 (1992). 
\bibitem[Ka]{kac} V. Kac. Lie Superalgebras.  {\it Adv. Math.} 26, 8-96 (1977)
\bibitem[Ki]{K} Kitchloo, N. The stable symplectic category and quantization.  arXiv:1204.5720; Contemporary Math. (AMS), Vol. 620, 251-280, 2014.
\bibitem[M]{M} O.~Mathieu, Harmonic cohomology classes of symplectic manifolds.  {\em Comment. Math. Helv.} 70, 1-9 (1995)
\bibitem[S]{S} Jedrzej \'Sniatycki, 
\emph{On cohomology groups appearing in geometric quantization},
in: Lecture Notes in Math.\ \textbf{570}, Springer-Verlag, Berlin, 1977, pp.~46--66.
\bibitem[TY]{TY} L.-S.~Tseng and S.-T.~Yau, Cohomology and Hodge theory on symplectic manifolds I.
J. Differential Geom. 91(3): 383-416 (July 2012).

--Cohomology and Hodge theory on symplectic manifolds II.  {\em J. Diff. Geom.} 91, 417-444 (2012)

\bibitem[WW]{WW} K. Wehrheim, C. Woodward.  Functoriality for Lagrangian correspondences in Floer theory.  arXiv:0708.2851
 \bibitem[W]{Weinstein} A.~Weinstein, \emph{Lectures on Symplectic Manifolds},  American Mathematical Society Regional Conference Series in Mathematics, Number 29, 1977.
  
  -- Symplectic Geometry. Bulletin of the American Mathematical Society, 5 (1981), 1-13
  
  -- Symplectic Categories.  arXiv 0911.4133.
  
  \bibitem[Wel]{Wells}  R.O. Wells, {\it Differential analysis on complex manifolds.}  Springer, 1980.

\end{thebibliography}
\end{document}